\documentclass[a4paper, 11pt]{amsart}
\title[Cohomological $\chi$-independence]
{Cohomological $\chi$-independence for Higgs bundles and Gopakumar--Vafa invariants}
\date{}
\author{Tasuki Kinjo and Naoki Koseki}

\address{The University of Liverpool, Mathematical Sciences Building, Liverpool, L69 7ZL, UK.}
\email{koseki@liverpool.ac.uk}

\address{graduate school of mathematical science, the university of tokyo, 3-8-1 komaba,
meguroku, tokyo 153-8914, japan.}
\email{tasuki.kinjo@ipmu.jp}

\makeatletter
 
  \@addtoreset{equation}{section}
 \makeatother

\usepackage{amsmath, amssymb, amsthm, amscd, comment, mathtools, amsfonts, color}
\usepackage{stmaryrd}
\usepackage[frame,cmtip,curve,arrow,matrix,line,graph]{xy}
\usepackage[mathscr]{euscript}
\usepackage{mathrsfs}
\usepackage{tikz}
\usetikzlibrary{intersections, calc}

\usepackage[T1]{fontenc}
\usepackage{hyperref}

\theoremstyle{plain}
\newtheorem{thm}{Theorem}[section]
\newtheorem{prop}[thm]{Proposition}
\newtheorem{def-prop}[thm]{Definition-Proposition}
\newtheorem{lem}[thm]{Lemma}
\newtheorem{cor}[thm]{Corollary}
\newtheorem*{thm*}{Theorem}

\theoremstyle{definition}
\newtheorem{defin}[thm]{Definition}
\newtheorem{conj}[thm]{Conjecture}
\newtheorem*{NaC}{Notation and Convention}
\newtheorem*{ACK}{Acknowledgement}

\newtheorem{rmk}[thm]{Remark}

\newtheorem{ex}[thm]{Example}

\DeclareMathOperator{\rk}{rk}

\DeclareMathOperator{\Spec}{Spec}
\DeclareMathOperator{\id}{id}
\newcommand{\dR}{\mathbf{R}}

\newcommand{\bA}{\mathbb{A}}

\newcommand{\bC}{\mathbb{C}}
\newcommand{\bD}{\mathbb{D}}

\newcommand{\bL}{\mathbb{L}}
\newcommand{\bN}{\mathbb{N}}
\newcommand{\bQ}{\mathbb{Q}}
\newcommand{\bR}{\mathbb{R}}
\newcommand{\bS}{\mathbb{S}}
\newcommand{\bZ}{\mathbb{Z}}
\newcommand{\mcA}{\mathcal{A}}

\newcommand{\mcC}{\mathcal{C}}

\newcommand{\mcE}{\mathcal{E}}
\newcommand{\mcF}{\mathcal{F}}

\newcommand{\mcH}{\mathcal{H}}

\newcommand{\mcO}{\mathcal{O}}

\newcommand{\mcS}{\mathcal{S}}

\newcommand{\mcW}{\mathcal{W}}
\newcommand{\mcX}{\mathcal{X}}

\DeclareMathOperator{\Hom}{Hom}
\DeclareMathOperator{\Tot}{Tot}
\DeclareMathOperator{\Coh}{Coh}
\DeclareMathOperator{\Supp}{Supp}
\DeclareMathOperator{\Ker}{Ker}

\DeclareMathOperator{\Ext}{Ext}

\DeclareMathOperator{\Sym}{Sym}
\DeclareMathOperator{\IC}{IC}
\DeclareMathOperator{\GW}{GW}
\DeclareMathOperator{\Dol}{Dol}
\DeclareMathOperator{\Betti}{B}

\DeclareMathOperator{\cl}{cl}
\DeclareMathOperator{\codim}{codim}
\DeclareMathOperator{\GL}{GL}

\DeclareMathOperator{\Quot}{Quot}
\DeclareMathOperator{\Chow}{Chow}
\DeclareMathOperator{\Perv}{Perv}

\DeclareMathOperator{\pt}{pt}

\DeclareMathOperator{\HN}{HN}

\DeclareMathOperator{\pr}{pr}

\DeclareMathOperator{\im}{im}
\DeclareMathOperator{\tr}{tr}
\DeclareMathOperator{\IH}{IH}
\DeclareMathOperator{\MMHM}{MMHM}
\DeclareMathOperator{\mhm}{mhm}
\DeclareMathOperator{\mmhm}{mmhm}
\DeclareMathOperator{\PG}{PG}

\DeclareMathOperator{\reg}{reg}
\DeclareMathOperator{\Jac}{Jac}

\DeclareMathOperator{\mH}{H}

\DeclareMathOperator{\RHom}{RHom}

\DeclareMathOperator{\Map}{Map}

\DeclareMathOperator{\rank}{rank}
\DeclareMathOperator{\rat}{rat}

\DeclareMathOperator{\vir}{vir}

\DeclareMathOperator{\Crit}{Crit}
\DeclareMathOperator{\red}{red}
\DeclareMathOperator{\vdim}{vdim}
\DeclareMathOperator{\MHM}{MHM}
\DeclareMathOperator{\BPS}{BPS}
\DeclareMathOperator{\Sch}{Sch}

\DeclareMathOperator{\sm}{sm}
\DeclareMathOperator{\semist}{ss}
\DeclareMathOperator{\sep}{sep}
\DeclareMathOperator{\mult}{mult}

\newcommand\DCrit{\mathop{\mathbf{Crit}}\nolimits}
\newcommand\DSpec{\mathop{\mathbf{Spec}}\nolimits}

\newcommand{\HBM}{\mH^{\mathrm{BM}}}
\newcommand{\sHom}{\mathop{\mathcal{H}\! \mathit{om}}\nolimits}
\newcommand{\sBPS}{\mathop{\mathcal{BPS}}\nolimits}
\newcommand{\sIC}{\mathop{\mathcal{IC}}\nolimits}

\renewcommand{\ss}{\mathrm{ss}}

\newcommand\ddr{d_{\mathrm{dR}}}
\newcommand\ddrcl{d_{\mathrm{dR}}^{\, \cl}}

\newcommand{\cA}{{\mathcal A}}

\newcommand{\cF}{{\mathcal F}}

\newcommand{\cH}{{\mathcal H}}

\newcommand{\cO}{{\mathcal O}}
\newcommand{\cP}{{\mathcal P}}

\newcommand{\cS}{{\mathcal S}}

\newcommand{\fM}{{\mathfrak M}}

\newcommand{\fS}{{\mathfrak S}}

\newcommand{\fU}{{\mathfrak U}}

\newcommand{\fW}{{\mathfrak W}}
\newcommand{\fX}{{\mathfrak X}}
\newcommand{\fY}{{\mathfrak Y}}
\newcommand{\fZ}{{\mathfrak Z}}

\newcommand{\sR}{{\mathscr{R}}}

\newcommand\bs{\boldsymbol}
\newcommand{\bfT}{\mathbf{T}}

\newcommand\freefootnote[1]{%
  \let\thefootnote\relax%
  \footnotetext{#1}%
  \let\thefootnote\svthefootnote%
}

\begin{document}

\begin{abstract}
The aim of this paper is two-fold:
Firstly, we prove Toda's $\chi$-independence conjecture for Gopakumar--Vafa invariants of arbitrary local curves. 
Secondly, following Davison's work, 
we introduce the BPS cohomology for moduli spaces of Higgs bundles of rank $r$ and Euler characteristic $\chi$ which are not necessary coprime, 
and show that it does not depend on $\chi$. 
This result extends the Hausel--Thaddeus conjecture on the $\chi$-independence of E-polynomials proved by Mellit,  Groechenig--Wyss--Ziegler and Yu in two ways: we obtain an isomorphism of mixed Hodge modules on the Hitchin base rather than an equality of E-polynomials, and we do not need the coprime assumption.

The proof of these results is based on a description of the moduli stack of one-dimensional coherent sheaves on a local curve as a global critical locus which is obtained in the companion paper by the first author and Naruki Masuda.

\end{abstract}

\maketitle

\setcounter{tocdepth}{1}
\tableofcontents

\section{Introduction}

We work over the complex number field $\bC$.

\subsection{Motivation and Results}

\subsubsection{Non-abelian Hodge theory}

Let $C$ be a smooth projective curve with genus $g \geq 2$.
By the non-abelian Hodge correspondence 
due to \cite{Cor88, Don87, Hit87, Sim92}, 
there is a homeomorphism 
\begin{equation} \label{eq:nonab}
M_{\Dol}(r, m) \simeq M_B(r, m), 
\end{equation}
where $M_{\Dol}(r, m)$ is the moduli space of 
slope semistable Higgs bundles $E$ on $C$ 
with $\rank(E)=r, \chi(E)=m$, 
and $M_B(r, m)$ is the twisted character variety, i.e., the quotient variety
\[
M_{\Betti}(r, m) \coloneqq 
\Bigl\{ A_i, B_i \in \GL_r, i = 1, \ldots, g \mid \prod_i [A_i, B_i] = e^{2 m \pi \sqrt{-1} / r} 
\Bigr\} \sslash \GL_r
\]
by the conjugate $\GL_r$ action. 

Assume for a while that 
$r$ and $m$ are coprime 
so that the moduli spaces in \eqref{eq:nonab}
are smooth. 
The homeomorphism \eqref{eq:nonab} induces an isomorphism 
\begin{equation} \label{eq:Hnonab}
    \mH^*(M_{\Dol}(r, m) ) \simeq
    \mH^*(M_B(r, m) ) 
\end{equation}
between the singular cohomology groups. 
However, since \eqref{eq:nonab} is only a diffeomorphism, 
the isomorphism \eqref{eq:Hnonab} is {\it not} an isomorphism of 
mixed Hodge structures. 
Indeed, the mixed Hodge structure on $\mH^*(M_{\Dol}(r, m) )$ is pure, 
while that on $\mH^*(M_B(r, m) )$ is not pure. 
Instead, the cohomology group 
$\mH^*(M_{\Dol}(r, m) )$ 
has the so-called {\it perverse filtration} 
induced by the Hitchin morphism 
\[
h \colon M_{\Dol}(r, m) \to B. 
\]
De Cataldo--Hausel--Migliorini \cite{CHM12}
conjectured that 
the perverse filtration on $\mH^*(M_{\Dol}(r, m) )$ 
matches with the weight filtration on $\mH^*(M_B(r, m) )$
via the isomorphism \eqref{eq:Hnonab} 
({\it P=W conjecture}). 
This conjecture was  recently proved by 
Hausel--Mellit--Minets--Schiffmann \cite{hmms22} and 
Maulik--Shen \cite{ms22} independently.


For character varieties, 
$M_B(r, m)$ and $M_B(r, m')$ are Galois conjugate to each other, 
for all $m, m' \in \bZ$ with $\gcd(r, m)=\gcd(r, m')=1$. 
In particular, we have an isomorphism 
\begin{equation}\label{eq:chiBetti}
\mH^*(M_{\Betti}(r, m) ) \cong
\mH^*(M_{\Betti}(r, m') ) 
\end{equation}
of mixed Hodge structures. 
According to the P=W conjecture, 
the perverse filtration on $\mH^*(M_{\Dol}(r, m))$ 
should be independent of $m \in \bZ$, 
as long as we have $\gcd(r, m)=1$. 
We prove this statement using the cohomological Donaldson--Thomas theory.

\begin{thm}[Example \ref{ex:ICvsBPS}]\label{thm:introchicoprime}
Let $r, m, m'$ be integers such that $r>0$ and $\gcd(r, m) = \gcd(r, m') = 1$ hold.
Then there exists an isomorphism
\[
\mH^*(M_{\Dol}(r, m) ) \cong
\mH^*(M_{\Dol}(r, m'))  
\]
preserving the Hodge structure and the perverse filtration.
\end{thm}
This kind of statement is called a {\it$\chi$-independence phenomenon}, as an invariant of the moduli space of Higgs bundles depends only on the rank $r$ and independent of the choice of the Euler characteristic $\chi$. 
Note that the above result for the perverse filtration was obtained by \cite{CMSZ21} 
independently, via a completely different method. 

Now assume that $(r, m)$ is not coprime.
In this case, the moduli spaces 
$M_{\Dol}(r, m)$ and $M_{\Betti}(r, m)$ are singular. 
Hence it is not clear which cohomology theory is 
a right choice to obtain a P=W type statement.
There are two candidates for this:
\begin{enumerate}
    \item Intersection cohomology (\cite{CM18, FM20, Mau21, FSY21}).
    \item BPS cohomology (\cite{CDP14}).
\end{enumerate}
One advantage of using the intersection cohomology is that it is mathematically defined whereas the BPS cohomology is defined in the physical language. Instead of this, BPS cohomology has its own advantage: whereas the $\chi$-independence phenomena for the intersection cohomology is only expected when we have $\gcd(r, m) = \gcd(r, m')$, the $\chi$-independence for the BPS cohomology is expected to hold without any assumption. Further, the BPS cohomology groups in both sides are expected to carry a Lie algebra structure (see \cite{Dav20}) and the non-abelian Hodge correspondence \eqref{eq:nonab} is expected to induce an isomorphism of these Lie algebras \cite[Conjecture 1.5]{SS20}.
This suggests that we would have a representation theoretic approach to the original P=W conjecture.

Following Davison's idea \cite{Dav17b}, we propose a definition of the BPS cohomology for the Dolbeault moduli space $\mH_{\BPS}^*(M_{\Dol}(r, m))$ as a cohomology of a pure Hodge module $\sBPS_{r, m}$ on $M_{\Dol}(r, m)$ defined using the cohomological Donaldson--Thomas theory (or refined BPS state counting) for $\Tot_{C}(\cO_C \oplus \omega_C)$. 
We have a split injection $\sIC_{M_{\Dol}(r, m)} \hookrightarrow \sBPS_{r, m}$ which is an isomorphism when $\gcd(r, m) = 1$, but not necessarily so for general $(r, m)$.
We prove the following $\chi$-independence for the BPS cohomology, which is a non-coprime generalization of Theorem \ref{thm:introchicoprime}:

\begin{thm}[Corollary \ref{cor:chiBPS}]\label{thm:introchiBPS}
Let $r, m, m'$ be integers such that $r>0$.
Then there exists an isomorphism
\[
\mH_{\BPS}^*(M_{\Dol}(r, m)) \cong \mH_{\BPS}^*(M_{\Dol}(r, m'))
\]
preserving the Hodge structure and the perverse filtration.
\end{thm}

\begin{rmk}
When we have $\gcd(r, m) = \gcd(r, m')$,
the Betti moduli spaces $M_{\Betti}(r, m)$ and $M_{\Betti}(r, m')$ are Galois conjugate.
Therefore we expect that there exists an isomorphism
\begin{equation}\label{eq:chiBetti2}
\mH_{\BPS}^*(M_{\Betti}(r, m)) \cong \mH_{\BPS}^*(M_{\Betti}(r, m'))
\end{equation}
preserving the mixed Hodge structure, though we do not discuss the definition of the BPS cohomology for the Betti moduli spaces in this paper.
Therefore Theorem \ref{thm:introchiBPS} gives an evidence of the P=W conjecture for the BPS cohomology.
Conversely, P=W conjecture and Theorem \ref{thm:introchiBPS} suggest that the isomorphism \eqref{eq:chiBetti2} holds without the assumption 
$\gcd(r, m) = \gcd(r, m')$, which is of independent interest.
\end{rmk}

\begin{rmk}
Recently, Davison--Hennecart--Schlegel-Mejia \cite{DHM22} established a theorem relating the BPS cohomology and the intersection cohomology for the moduli space of Higgs bundles and for the character varieties.
Their work imply the equivalence of two versions of the P=W conjectures via the BPS cohomology and via the intersection cohomology, and that the $\chi$-independence of the intersection cohomology of the Dolbeault moduli space  follows from Theorem \ref{thm:introchiBPS} as long as $\gcd(r, m) = \gcd(r, m')$ holds.
\end{rmk}

We also establish the cohomological integrality theorem for Higgs bundles,
which claims the decomposition of the 
Borel--Moore homology of the moduli stack of Higgs bundles $\HBM_*(\fM_{\Dol}(r, m))$ into tensor products of the BPS cohomology (see Theorem \ref{thm:intHiggs2} for the precise statement).
A similar statement was proved for quivers with potentials in
\cite[Theorem A]{DM20} 
 and for preprojective algebras in \cite[Theorem D]{Dav17b}. 
As explained in \cite[\S 6.7]{DM15}, a plethystic computation 
and the cohomological integrality theorem imply that the Euler characteristic of the BPS cohomology is equal to the genus zero BPS invariant introduced by  Joyce--Song \cite[\S 6.2]{JS12}. In particular, cohomological integrality theorem strengthens the integrality conjecture for the genus zero BPS invariants \cite[Conjecture 6.12]{JS12}. 

Combining the cohomological integrality theorem and the $\chi$-independence theorem (Theorem \ref{thm:introchiBPS}), we obtain the following $\chi$-independence result for the Borel--Moore homology of the moduli stack $\fM_{\Dol}(r, m)$ of Higgs bundles:

\begin{cor}[Corollary \ref{cor:chiBM}]\label{cor:introchistack}
Let $r, m, m'$ be integers such that $r>0$ and $\gcd(r, m) = \gcd(r, m')$ hold. Then there exists an isomorphism 
\[
\HBM_*(\fM_{\Dol}(r, m)) \cong \HBM_*(\fM_{\Dol}(r, m'))
\]
preserving the Hodge structure and the perverse filtration introduced in \cite[Proposition 7.24]{Dav21}.
\end{cor}

\subsubsection{Gopakumar--Vafa (BPS) invariants}

More generally, 
we investigate the $\chi$-independence phenomena 
for curve counting theory on 
a class of Calabi--Yau (CY) 3-folds 
called {\it local curves}. 
By definition, a local curve 
is a CY 3-fold of the form $\Tot_C(N)$, 
where $C$ is a smooth projective curve 
and $N$ is a rank $2$ vector bundle on $C$ 
such that $\det(N) \cong  \omega_C$. 
To explain our result, 
we recall some basic background of 
curve counting theory for CY 3-folds.

There are several ways to count curves in a CY 3-fold $X$, 
and one of them is the {\it Gromov--Witten (GW) theory}: 
For an integer $g \geq 0$ and 
a homology class $\beta \in H_2(X, \bZ)$, 
denote by $M_{g, \beta}(X)$ the moduli space of 
stable maps $f \colon C \to X$ with 
$C$ nodal curves of arithmetic genus $g$ and 
$f_{*}[C]=\beta$. 
Then the GW invariant is defined as 
\[
\GW_{g, \beta}\coloneqq\int_{[M_{g, \beta}(X)]^{\vir}}1, 
\] 
where $[M_{g, \beta}(X)]^{\vir}$ denotes 
the virtual fundamental cycle. 
Due to the existence of stacky points in the moduli space 
$M_{g, \beta}(X)$, the GW invariant $\GW_{g, \beta}$ is 
in general a rational number. 

Based on string theory, Gopakumar--Vafa \cite{GV98} 
conjectured the existence of {\it integer valued} invariants 
$n_{g, \beta} \in \bZ$ for $g \geq 0$ and $\beta \in H_2(X, \bZ)$, 
satisfying the equation 
\begin{equation} \label{eq:GW=GV}
\sum_{g \geq 0, \beta >0}\GW_{g, \beta} \lambda^{2g-2} t^\beta
=\sum_{g \geq 0, \beta>0, k \geq 1} \frac{n_{g, \beta}}{k}\left(
2\sin\left(
\frac{k\lambda}{2}
\right)
\right)^{2g-2}t^{k\beta}. 
\end{equation}
We call the invariants $n_{g, \beta}$ 
the {\it Gopakumar--Vafa (GV) invariants} 
(also known as the {\it BPS invariants}). 

Building on the previous works by 
Hosono--Saito--Takahashi \cite{HST01} 
and Kiem--Li \cite{KL16}, 
Maulik--Toda \cite{MT18} and Toda \cite{Tod17} 
proposed the mathematical definition of the GV invariants. 
Following the original idea of Gopakumar--Vafa, 
they consider the moduli space $M_X(\beta, m)$ 
of slope semistable one-dimensional sheaves $E$ on $X$ satisfying 
$[E]=\beta \in H_2(X, \bZ)$ and $\chi(E)=m \in \bZ$. 
The moduli space $M_X(\beta, m)$ admits 
the {\it Hilbert--Chow} morphism 
\[
\pi_M \colon M_X(\beta, m)^{\red} \to \Chow_\beta(X),
\]
which sends a sheaf to its fundamental cycle. 
Maulik--Toda \cite{MT18} and 
Toda \cite{Tod17} defined 
the {\it generalized GV invariants} by the formula 
\begin{equation} \label{defin:introGV}
\sum_{i \in \bZ}\chi\left(
{\mcH}^i(\pi_{M*}\left(\varphi_{M_X(\beta, m)}) \right)
\right)y^i=
\sum_{g \geq 0}n_{g, \beta, m}\left(
y^{\frac{1}{2}}+y^{-\frac{1}{2}}
\right)^{2g}, 
\end{equation}
where $\varphi_{M_X(\beta, m)}$ is a certain perverse sheaf on $M_X(\beta, m)$, 
see Sections \ref{ssec:van} and \ref{sec:GVdefin} 
for more detail. 

As the GV invariants are conjecturally equivalent to the GW invariants 
by the formula \eqref{eq:GW=GV}, 
the GV invariants should be independent of 
the additional choice of the Euler characteristic 
$m \in \bZ$: 

\begin{conj}[{\cite[Conjecture 1.2]{Tod17}}] 
\label{conj:introchi}
The generalized GV invariants are independent of the choice of $m \in \bZ$, 
i.e., we have 
\[
n_{g, \beta, m}=n_{g, \beta, m'}
\]
for all $m, m' \in \bZ$. 
\end{conj}
We call the above conjecture as 
{\it $\chi$-independence conjecture} 
for GV invariants. 
In this paper, we prove it 
for local curves in full generality: 
\begin{thm}[Theorem \ref{thm:chi-lC}] 
\label{thm:introchi}
Conjecture \ref{conj:introchi} holds 
for $X=\Tot_C(N)$. 
\end{thm}

\subsection{Strategy of the proof}
\subsubsection{Results on local curves}
The key ingredient in our arguments is the main result of the companion paper by the first author and Masuda \cite{KM21} on the construction of 
a {\it global d-critical chart} 
for the moduli space $\fM_X \coloneqq \fM_X(\beta, m)$ 
of one-dimensional semistable sheaves on a local curve $X=\Tot_C(N)$, 
i.e., the description of the moduli space $\fM_X$ as 
the critical locus inside a certain smooth space: 
Take an exact sequence 
\begin{equation} \label{eq:introex}
    0 \to L_1 \to N \to L_2 \to 0, 
\end{equation}
where $L_1, L_2$ are line bundles 
with $\deg(L_2)$ sufficiently large. 
We denote by $Y \coloneqq \Tot_C(L_2)$. 
Then it is shown in \cite[Theorem 5.6]{KM21} that there exists a function 
$f \colon M_Y \to \bA^1$ 
on the good moduli space of one-dimensional semistable sheaves on $Y$ 
such that we have an isomorphism 
\[
\fM_X \cong \{d(f \circ p_Y)=0\} \subset \fM_Y, 
\]
where $p_Y \colon \fM_Y \to M_Y$ is the natural map from the moduli stack of one-dimensional semistable sheaves to its good moduli space.

In this situation, the perverse sheaf 
appeared in the definition \eqref{defin:introGV}
of the generalized GV invariants 
coincides with the vanishing cycle sheaf: 
\[
\varphi_{M_X} \cong \varphi_f(\sIC_{M_Y}), 
\]
and the proof of Theorem \ref{thm:introchi} 
is reduced to proving the corresponding statement 
for the intersection complex $\sIC_{M_Y}$. 
The latter is proved 
in the recent paper by Maulik--Shen \cite{MS20}, 
hence we obtain Theorem \ref{thm:introchi}.

\subsubsection{Results on Higgs bundles}
We define the BPS sheaf $\sBPS_{r, m}$ on the moduli space $M_{\Dol}(r, m)$ using the vanishing cycle complex $\varphi_{M_X}$ for $X = \Tot_C(\cO_C \oplus \omega_C)$.
Then the argument as above also implies
Theorem \ref{thm:introchiBPS}.
The cohomological integrality theorem for Higgs bundles (Theorem \ref{thm:intHiggs2}) is obtained by extending the argument for quivers with potentials \cite[Theorem A]{DM20} using the global critical locus description of $\fM_X$ and applying the first author's dimensional reduction theorem \cite[Theorem 4.14]{Kin21} which relates the vanishing cycle cohomology for $\fM_X$ and the Borel--Moore homology for $\fM_{\Dol}$.

\subsection{Relation with existing works}
\begin{enumerate}

\item Mellit \cite{Mel20}, 
Groechenig--Wyss--Ziegler \cite{GWZ20}, 
and Yu \cite{Yu21} 
proved that the E-polynomial of $M_{\Dol}(r, m)$ 
is independent of $m \in \bZ$ when $\gcd(r, m)=1$. 
These results were proved 
via the reduction to the positive characteristics. 

We extended the result to the non-coprime case and further lifted the equality to an isomorphism of Hodge structures
via the completely different methods.

\item Recently, de Cataldo--Maulik--Shen--Zhang \cite{CMSZ21} used a positive characteristic method to prove that the isomorphism \eqref{eq:chiBetti} preserves the perverse filtration induced by the non-abelian Hodge theorem.

At present, we do not know whether our cohomological $\chi$-independence results are compatible with the Galois conjugate.

\item Toda \cite{Tod17} proved 
Conjecture \ref{conj:introchi}  for 
primitive classes $\beta \in H_2(X, \bZ)$ 
(assuming a technical conjecture on orientation data). 
For non-primitive classes, Maulik--Shen \cite{MS20} proved it 
for local toric del Pezzo surfaces 
and recently \cite{Yua21} removed the toric assumption from their result.
Maulik--Shen \cite{MS20} also proved the conjecture for local curves of the form 
$\Tot_C(\mcO(D) \oplus \omega_C(-D))$ 
for a divisor $D$ with $\deg(D) > 2g(C)-2$. 

Our Theorem \ref{thm:introchi} proves 
Conjecture \ref{conj:introchi} for {\it arbitrary} local curves. 
In particular, the result for $X=\Tot_C(N)$ with indecomposable $N$ 
is completely new. 
\end{enumerate}


\subsection{Structure of the paper}
The paper is organized as follows. 
In Section \ref{sec:preGV}, 
we recall Joyce's theory on d-critical structures. 
Then we recall the definition of the GV invariants, 
and introduce the notion of local curves and twisted Higgs bundles.

In Section \ref{sec:chi}, we prove Theorem \ref{thm:introchi}. 
In Section \ref{sec:int}, 
we prove the cohomological integrality theorem 
for $D$-twisted Higgs bundles where $\deg(D)>2g-2$, 
which plays an important role in the proofs of 
Theorems \ref{thm:introchi} and Corollary \ref{cor:introchistack}. 
Finally in Section \ref{sec:Higgs}, 
we discuss applications to Higgs bundles. 
We prove Theorem \ref{thm:introchiBPS} and the cohomological integrality theorem for Higgs bundles (Theorem \ref{thm:intHiggs2}). 

In Appendix \ref{app:tech}, we give a brief overview of the shifted symplectic geometry and prove some technical lemmas that we use in this paper.

In Appendix \ref{app:supplem}, we prove a version of the support lemma of the vanishing cycle complexes which is needed to define the BPS sheaf.

\begin{ACK}
The authors would like to thank 
Professors Ben Davison, Yukinobu Toda and Junliang Shen
for fruitful discussions and for carefully reading the previous version of this article. 
The first author would like to thank Naruki Masuda for the collaboration on the companion paper \cite{KM21}.
The second author would like to thank 
Professors Arend Bayer and Jim Bryan 
for related discussions. 

T.K. was supported by WINGS-FMSP
program at the Graduate School of Mathematical Science, the University of Tokyo and JSPS KAKENHI Grant number JP21J21118. 
N.K. was supported by 
ERC Consolidator grant WallCrossAG, no.~819864. 
\end{ACK}

\begin{NaC}
In this paper, we work over the complex number field $\bC$. 
We use the following notations: 
\begin{itemize}
     \item We let $\mathbb{S}$ denote the $\infty$-category of spaces (see \cite[Definition 1.2.16.1]{HTT}).

    \item We basically write stacks in Fraktur (e.g. $\fX, \fY, \ldots$), and derived schemes, derived stacks and morphisms between derived stacks in bold (e.g. $\bs{X}, \bs{\fX}, \bs{f}, \ldots$).
    We will write $X = t_0(\bs{X})$, $\fX = t_0(\bs{\fX})$, $f = t_0(\bs{f})$ and so on.

\item A derived Artin stack $\mathbb{\fX}$ is said to be \textbf{quasi-smooth}
if the cotangent complex $\mathbb{L}_{\fX}$ has Tor-amplitude $[-1, 1]$.

\item All derived/underived Artin stacks are assumed to have quasi-compact and separated diagonals and locally finitely presented over the complex number field.
As the fiber product of finite type separated schemes over such a stack is again of finite type and separated,
we can consider the category of mixed Hodge modules on such stacks
(see \S \ref{ssec:MHM} for the detail).

\item For a separated complex analytic space $X$, we let $D_c^b(X)$ denote the bounded derived category of complexes of sheaves in $\bQ$-vector spaces on $X$ with constructible cohomology. 

\item For a complex analytic stack $\fX$, we let $D^b_c(\fX)$ denote the bounded derived category of sheaves in $\bQ$-vector spaces on $\fX_{\textrm{lis-an}}$ with constructible cohomology. Here $\fX_{\textrm{lis-an}}$ denote the lisse-analytic topos of $\fX$ (see \cite[\S 3.2.3]{Sun17}).

\item If there is no confusion, we use expressions such as $f_*$ and $f_!$ for the
derived functors $\dR f_*$ and $\dR f_!$.

\end{itemize}
\end{NaC}

\section{Preliminaries} 

In this section, we collect some basic notions that we use in this paper.
Firstly we recall Joyce's theory of d-critical locus in \S \ref{ssec:dcrit}. Then we review the construction of vanishing cycle complexes associated with d-critical stacks in \S \ref{ssec:van}.
In \S \ref{sec:GVdefin} we review Maulik--Toda's construction \cite{MT18} of Gopakumar--Vafa invariants based on vanishing cycle complexes. In \S \ref{sec:local} we collect some basic facts on local curves and recall Maulik--Shen's cohomological $\chi$-independence theorem \cite{MS20}.

\label{sec:preGV}

\subsection{D-critical structures}\label{ssec:dcrit}

In \cite{Joy15}, Joyce introduced the notion of d-critical structures which are classical models of derived critical loci. We now briefly recall it.

Let $X$ be a complex analytic space.
Joyce \cite[Theorem 2.1]{Joy15} introduced a sheaf $\cS_X$ on $X$ with the following property: for an open subset $R \subset X$ and an embedding $i \colon R \hookrightarrow U$ to a complex manifold $U$, there exists a short exact sequence
\[
0 \to \cS_X |_{R} \to i^{-1}\cO_U /I_{R, U}^2 \ \xrightarrow[]{\ddr} i^{-1}\Omega_U/(I_{R, U} \cdot i^{-1}\Omega_U),
\]
where $I_{R, U}$ is the ideal sheaf of $R$ in $U$.
One can show that the natural map 
\[
\cS_{X}|_R \to i^{-1}\cO_U /I_{R, U}^2 \twoheadrightarrow \cO_R
\]
glues to define a morphism $\cS_X \to \cO_X$. We define a subsheaf $\cS_X^0 \subset \cS_X$ by the kernel of the map
\[
\cS_X \to \cO_X \twoheadrightarrow \cO_{X^{\red}}.
\]

If $X$ is the critical locus $\Crit(f)$ of a holomorphic function $f$ on a complex manifold $U$ such that $f |_{X^{\red}} = 0$, 
then $f + I_{X, U}^{2}$ defines an element of $\cS_{X}^0$.

\begin{defin}
Let $X$ be a complex analytic space. A section $s \in \Gamma(X, \cS_X^{0})$ is called a \textbf{d-critical structure} if for each point $x \in X$, there exists an open neighborhood $R \subset X$, an embedding $i \colon R \hookrightarrow U$ into a complex manifold, and a holomorphic function $f$ on $U$ with the property $f |_{R^{\red}} = 0$ such that $f + I_{R, U}^2 = s |_R$.
The quadruple $(R, U, f, i)$ is called a \textbf{d-critical chart} of $X$. A complex analytic space equipped with a d-critical structure is called a \textbf{d-critical analytic space}.
\end{defin}

The sheaf $\cS_{X}^0$ has the following functorial property: for a given morphism of complex analytic spaces $q \colon X_1 \to X_2$, there exist natural morphisms
\[
q^{\star} \colon q^{-1} \cS_{X_2}^{0} \to \cS_{X_1}^0.
\]
Now assume that $q$ is smooth surjective and take a section $s \in \Gamma(X, \cS_X^0)$. Then it is shown in \cite[Proposition 2.8]{Joy15} that $q^{\star} s$ is a d-critical structure if and only if $s$ is a d-critical structure.

Now let $\fX$ be a complex analytic stack. Then it is shown in \cite[Corollary 2.52]{Joy15} that there exists a sheaf $\cS_{\fX}^0$ on the lisse-analytic site of $\fX$ with the following property:
\begin{itemize}
    \item For a smooth morphism $t \colon T \to \fX$, there exists a natural isomorphism $\eta_{t} \colon \cS_{\fX}^{0} |_{T} \cong \cS_T ^0$.
    \item For a morphism 
    \[
    q \colon (t_1 \colon T_1 \to \fX) \to (t_2 \colon T_2 \to \fX)
    \]
    between complex analytic spaces smooth over $\fX$, the natural map $q^{-1} (\cS_{\fX}^0 |_{T_2}) \to \cS_{\fX}^0 |_{T_1}$ is identified with $q^{\star}$. 
\end{itemize}
For a smooth morphism $t \colon T \to \fX$ from a scheme and a section $s \in \Gamma(\fX, \cS_{\fX}^0)$,
we write $t^{\star} s \coloneqq \eta_t(s|_{T}) \in \Gamma(T, \cS_{T}^0)$.

\begin{defin}
For a complex analytic stack $\fX$, a section $s \in \Gamma(\fX, \cS_{\fX}^0)$ is called a \textbf{d-critical structure} if for any smooth surjective morphism $t \colon T \to \fX$, the element $t^{\star}s$ is a d-critical structure on $T$. A \textbf{d-critical stack} is a complex analytic stack $\fX$ equipped with a d-critical structure.
\end{defin}

For a complex analytic stack $\fX$ equipped with a d-critical structure $s$, Joyce \cite[\S 2.4, \S 2.8]{Joy15} defines a line bundle $K_{\fX, s}^{\vir}$ on $\fX^{\red}$ called the \textbf{virtual canonical bundle} of $(\fX, s)$. If there is no confusion, we simply write $K_{\fX}^{\vir} = K_{\fX, s}^{\vir}$. We now recall some of its basic properties.
Firstly assume $\fX$ is a complex analytic space and write $\fX = X$.
Take a d-critical chart $\sR = (R, U, f, i)$ of  $(X, s)$.
Then there exists a natural isomorphism
\[
\iota_{\sR} \colon K_{X, s}^{\vir} |_{R^{\red}} \cong  K_U^{\otimes 2}  |_{R^{\red}}.
\]
Let $q \colon X_1 \to X_2$ be a smooth morphism and $s_2$ be a d-critical structure on $X_2$. Write $s_1 = q^{\star}s_2$. Then it is shown in \cite[Proposition 2.30]{Joy15} that there exists a natural isomorphism
\[
\Upsilon_q \colon q^{\red, *}K_{X_2, s_2}^{\vir} \otimes \det(\Omega_{X_1/X_2})|_{X_1^{\red}}^{\otimes 2} \cong K_{X_1, s_1}^{\vir}
\]
with the following property: if we are given d-critical charts $\sR_1 = (R_1, U_1, f_1, i_1)$ of $(X_1, s_1)$ and $\sR_2 = (R_2, U_2, f_2, i_2)$ of $(X_2, s_2)$ such that $q(R_1) \subset R_2$, and a smooth morphism $\tilde{q} \colon U_1 \to U_2$ such that $f_1 = f_2 \circ \tilde{q}$ and $i_2 \circ q|_{R_1} = \tilde{q} \circ i_1$, the following diagram of line bundles on $R_1^{\red}$ commutes:
\[
\xymatrix{
 q^{\red, *}K_{X_2, s_2}^{\vir}|_{R_1^{\red}} \otimes \det(\Omega_{X_1/X_2})|_{R_1^{\red}}^{\otimes 2} \ar[r]^-{\Upsilon_q |_{R_1^{\red}} } \ar[d]^-{(q|_{R_1})^{\red, *} \iota_{\sR_2} \otimes \id}
 & K_{X_1, s_1}^{\vir} |_{R_1^{\red}} \ar[d]^-{\iota_{\sR_1}} \\
 \tilde{q}^* K_{U_2}^{\otimes {2}} |_{R_1^{\red}} \otimes \det(\Omega_{X_1/X_2})|_{R_1^{\red}}^{\otimes 2} \ar[r]
 & K_{U_1}^{\otimes {2}} |_{R_1^{\red}}.
}
\]
Here the bottom horizontal arrow is defined by the natural isomorphism $\Omega_{X_1/X_2}|_{R_1} \cong \Omega_{U_1/U_2}|_{R_1}$.

Now we treat the stacky case. Let $\fX$ be a complex analytic space
and $t \colon T \to \fX$ be a smooth morphism from an analytic space.
Then there exists a natural isomorphism 
\[
\Upsilon_t \colon t^{\red, *}K_{\fX, s}^{\vir} \otimes \det(\Omega_{T/\fX})|_{T^{\red}}^{\otimes 2} \cong K_{T, t^{\star} s}^{\vir},
\]
see \cite[Theorem 2.56]{Joy15}. 
For a morphism 
    \[
    q \colon (t_1 \colon T_1 \to \fX) \to (t_2 \colon T_2 \to \fX)
    \]
between complex analytic spaces smooth over $\fX$, the following diagram commutes:
\begin{equation}\label{eq:diagUp}
\begin{split}
\xymatrix{
t_1^{\red, *}K_{\fX, s}^{\vir} \otimes \det(\Omega_{T_1/\fX})|_{T_1^{\red}}^{\otimes 2} \ar[d] \ar@/^40pt/[rdd]^-{\Upsilon_{t_1}}
&  \\
q^{\red, *} (t_2^{\red, *}K_{\fX, s}^{\vir} \otimes \det(\Omega_{T_2/\fX})|_{T_2^{\red}}^{\otimes 2}  ) \otimes \det(\Omega_{T_1/T_2})|_{T_1^{\red}}^{\otimes 2} \ar[d]^-{q^{\red, *} \Upsilon_{t_2} \otimes \id}
&  \\
q^{\red, *}K_{T_2, t_2^{\star} s}^{\vir} \otimes \det(\Omega_{T_1/T_2})|_{T_1^{\red}}^{\otimes 2} \ar[r]^-{\Upsilon_q}
& K_{T_1, t_1^{\star}s}^{\vir}.
}
\end{split}
\end{equation}

For a d-critical stack $(\fX, s)$, an \textbf{orientation} is a choice of a line bundle $L$ on $\fX^{\red}$ and an isomorphism $o \colon L^{\otimes {2}} \cong K_{\fX, s}^{\vir}$.
For a smooth morphism $t \colon T \to \fX$, we define an orientation
\[
t^{\star}o \colon (t^{\red, *}L \otimes \det(\Omega_{T/\fX})|_{T^{\red}})^{\otimes 2} \cong K_{T, t^{\star }s}^{\vir}
\]
using $\Upsilon_t$.
If we are given a smooth morphism $    q \colon (t_1 \colon T_1 \to \fX) \to (t_2 \colon T_2 \to \fX)$ between analytic spaces smooth over $\fX$, there exists a natural isomorphism
\begin{equation}\label{eq:oriass}
t_1^{\star}o \cong q^{\star} t_2^{\star}o.
\end{equation}

\subsection{Vanishing cycle complexes on d-critical stacks}\label{ssec:van}

In this subsection, we recall some basic properties of the vanishing cycle functors and the vanishing cycle complexes associated with oriented d-critical stacks.

Let $U$ be a complex manifold and $f$ be a holomorphic function on $U$.
Write $U_0 = f^{-1}(0)$.
Then the \textbf{vanishing cycle functor}
\[
\varphi_f \colon D^b_c(U) \to D^b_c(U_0).
\]
is defined by the following formula 
\[
\varphi_f \coloneqq (U_0 \hookrightarrow U_{\leq 0})^*(U_{\leq 0} \hookrightarrow U)^!,
\]
where $U_{\leq 0} \subset U$ denotes the closed subset ${\mathrm{Re}(f)}^{-1}(\bR_{\leq 0})$. It is shown in \cite[Corollary 10.3.13]{KS13} that the functor $\varphi_f$ preserves the perversity.
If there is no confusion, we write $\varphi_f \coloneqq \varphi_f(\bQ_U[\dim U])$.

Let $q \colon V \to U$ be a holomorphic map between complex manifolds. Write $V_0 \coloneqq (f \circ q)^{-1}(0)$ and we let $q_0 \colon V_0 \to U_0$ be the restriction of $q$.
By the definition of the vanishing cycle functor, we have the following base change morphisms
\begin{align*}
 \varphi_{f} \circ q_* &\to {q_0}_* \circ \varphi_{f \circ q}  \\
 q_0^* \varphi_f &\to \varphi_{f\circ q} \circ q^*.
\end{align*}
The first morphism is an isomorphism if $q$ is proper and the latter morphism is an isomorphism if $q$ is smooth. These are direct consequences of the proper/smooth base change theorem.

Now let $\fU$ be a smooth complex analytic stack and $f$ be a holomorphic function on $\fU$.
Write $\fU_0 \coloneqq f^{-1}(0)$.
For a perverse sheaf 
$\cP \in \Perv(\fX)$, we define the perverse sheaf 
\[
\varphi_f(\cP) \in \Perv(\fU_0)
\]
as follows: Take a smooth surjective morphism $q \colon U \to \fU$. We let $\pr_i \colon U \times_{\fU} U \to U$ denote the $i$-th projection and $\pr_{i, 0} \colon (f \circ q \circ \pr_i)^{-1}(0) \to (f \circ q)^{-1}(0)$ denote the restriction of $\pr_i$. Then we have a natural isomorphism
\begin{align*}
\pr_{1, 0}^* \varphi_{f \circ q}(q^*\cP) 
\cong \varphi_{f \circ q \circ \pr_1}(\pr_1^*q^*\cP) 
\cong \varphi_{f \circ q \circ \pr_2}(\pr_2^*q^*\cP)
\cong \pr_{2, 0}^* \varphi_{f \circ q}(q^*\cP).
\end{align*}
This isomorphism satisfies the cocycle condition, hence the shifted perverse sheaf $\varphi_{f \circ q}(q^* \cP)$ descends to a perverse sheaf $\varphi_f(\cP) \in \Perv(\fU_0)$. One can show that the construction does not depend on the choice of the smooth morphism $q$.

Now we recall the vanishing cycle complex associated with an oriented d-critical stack constructed in \cite[Theorem 4.8]{BBBBJ15}.

First we treat the non-stacky case. Let $(X, s, o)$ be an oriented d-critical analytic space. Then it is shown in \cite[Theorem 6.9]{BBDJS15} that there is a natural perverse sheaf 
\[
 \varphi_{X, s, o} \in \Perv(X)
\]
called the \textbf{vanishing cycle complex} associated with $(X, s, o)$. We sometimes omit $s$ and $o$ and write $\varphi_{X} = \varphi_{X, s, o}$ if there is no confusion.
For a d-critical chart $\sR = (R, U, f, i)$ of $(X, s)$,
we have a natural isomorphism
\[
\omega_{\sR} \colon \varphi_{X, s, o}|_R \cong i^* \varphi_f \otimes_{\bZ/2\bZ} Q_{\sR}^o,
\]
where $Q_{\sR}^o$ is a $\bZ/2\bZ$-local system on $R$ parametrizing local square roots of the isomorphism
\[
L^{\otimes 2} |_{R^{\red}} \xrightarrow[\cong]{o} K_{X, s}^{\vir} |_{R^{\red}} \cong i^* K_U^{\otimes 2}|_{R^{\red}}.
\]

\begin{ex}\label{ex:vanloc}
Let $U$ be a complex manifold and $f \colon U \to \bA^1$ be a holomorphic function such that $f|_{\Crit(f)^{\red}} =0$. Write $X = \Crit(f)$ and equip it with the canonical d-critical structure $s$ and the canonical orientation $o \colon K_U |_{X^{\red}}^{\otimes 2} \cong K_{X, s}^{\vir}$. Then $\sR \coloneqq (X, U, f, X \hookrightarrow U)$ defines a d-critical chart. In this case the local system $Q_{\sR}^o$ is trivial.
Therefore we have a natural isomorphism
\[
\varphi_{X, s, o} \cong \varphi_f |_X. 
\]
\end{ex}

Let $q \colon X_1 \to X_2$ be a smooth morphism and equip $X_2$ with a d-critical structure $s_2$ and an orientation $o_2$. Write $s_1 = q^{\star}s_2$ and $o_1 = q^{\star} o_2$. Then there exists a natural isomorphism of perverse sheaves
\[
\Theta_q \colon \varphi_{X_1,s_1, o_1} \cong q^* \varphi_{X_2, s_2, o_2}[\dim q]
\]
with the following property: 
If we are given d-critical charts $\sR_1 = (R_1, U_1, f_1, i_1)$ of $(X_1, s_1)$ and $\sR_2 = (R_2, U_2, f_2, i_2)$ of $(X_2, s_2)$ such that $q(R_1) \subset R_2$, and a smooth morphism $\tilde{q} \colon U_1 \to U_2$ such that $f_1 = f_2 \circ \tilde{q}$ and $i_2 \circ q|_R = \tilde{q} \circ i_1$, the following diagram in $\Perv(R_1)$ commutes:
\begin{equation}\label{eq:commvanpull}
\xymatrix@C=70pt{
\varphi_{X_1, s_1,  o_1}|_{R_1} \ar[r]^-{\omega_{\sR_1}} \ar[d]^-{\Theta_q |_{R_1}}
& i_1^* \varphi_{f_1} \otimes_{\bZ/2\bZ} Q_{\sR_1}^{o_1} \ar[d]  \\
(q|_{R_1})^*(\varphi_{X_2, s_2, o_2})[\dim q] \ar[r]^-{(q|_{R_1})^*\omega_{\sR_2}[\dim q]} 
& (q|_{R_1})^*(i_2^* \varphi_{f_2} \otimes_{\bZ/2\bZ} Q_{\sR_2}^{o_2})[\dim q],
}
\end{equation}
where the right vertical arrow is defined using the natural isomorphisms
$\varphi_{f_1} \cong (\tilde{q} |_{(f_2 \circ \tilde{q})^{-1}(0)})^* \varphi_{f_2}[\dim q]$ and $Q_{\sR_1}^{o_1} \cong (q|_{R_1})^* Q_{\sR_2}^{o_2}$.

Now we move to the stacky case. Let $(\fX, s, o)$ be a d-critical stack. Then it is shown in \cite[Theorem 4.8]{BBBBJ15} that there exists a natural perverse sheaf
\[
\varphi_{\fX, s, o} \in \Perv(\fX)
\]
with the following property:
If we are given a smooth morphism $t \colon T \to \fX$ from a complex analytic space, there exists a natural isomorphism
\[
\Theta_t \colon \varphi_{T, t^{\star}o, t^{\star}o} \cong t^* \varphi_{X, s, o}[\dim t].
\]
Furthermore, if we are given a smooth morphism $q \colon (t_1 \colon T_1 \to \fX) \to (t_2 \colon T_2 \to \fX)$ between schemes smooth over $\fX$,
the following diagram commutes:
\begin{equation}\label{eq:commvan}
\xymatrix{
\varphi_{T_1, t_1^{\star}s, t_1^{\star}o} \ar[r]^-{\sim} \ar[d]^-{\Theta_{t_1}}
& \varphi_{T_1, q^{\star}t_2^{\star}s, q^{\star}t_2^{\star}o} \ar[r]^-{\Theta_q}
& q^* \varphi_{T_2, t_2^{\star}s, t_2^{\star}o} \ar[d]^-{q^* \Theta_{t_2}}[\dim q] \\
t_1^* \varphi_{\fX, s, o}[\dim t_1] \ar[rr]^-{\sim}
& {}
& q^* t_2^* \varphi_{\fX, s, o}[\dim t_1].
}
\end{equation}

Let $\fU$ be a smooth Artin stack and $f \colon \fU \to \bA^1$ be a regular function on it.
Then it is shown in Example \ref{ex:crit2} that $\DCrit(f)$ carries a natural $(-1)$-shifted symplectic structure hence there exists a natural d-critical structure $s$ on its classical truncation $\fX \coloneqq \Crit(f)$.
We will see in Lemma \ref{lem:natori} that the d-critical analytic stack $(\fX, s)$ admits a canonical orientation
$o \colon K_{\fU}^{\otimes {2}} |_{\fX^{\red}} \cong K_{\fX, s}^{\vir}$.

\begin{prop} \label{prop:joyce=van}
There exists a natural isomorphism of perverse sheaves: 
\[
\theta \colon \varphi_{\fX, s, o} \cong \varphi_f(\bQ_{\fU}[\dim \fU]).
\]
\end{prop}

We postpone the proof to \S \ref{ssec:joyvan}.

\begin{rmk}
The argument in \cite[Theorem 4.8]{BBBBJ15} shows that the perverse sheaf $\varphi_{\fX, s, o}$ naturally extends to a mixed Hodge module $\varphi_{\fX, s, o}^{\mathrm{mhm}}$ and to a monodromic mixed Hodge module $\varphi_{\fX, s, o}^{\mathrm{mmhm}}$. Proposition \ref{prop:joyce=van} extends to an isomorphism of monodromic mixed Hodge modules with the same proof. We refer the reader to \S \ref{ssec:MMHM} for a brief discussion on monodromic mixed Hodge modules.
\end{rmk}

\subsection{Maulik--Toda's construction of Gopakumar--Vafa invariants}
\label{sec:GVdefin}
In this subsection, we recall 
the definition of generalized Gopakumar--Vafa (GV) invariants 
following \cite{MT18, Tod17}. 
Let $X$ be a smooth quasi-projective 
Calabi--Yau threefold and
$H$ be an ample divisor on $X$. 

\begin{defin} \label{def:slope}
Let $E$ be a pure one-dimensional coherent sheaf 
with compact support on $X$. 
\begin{enumerate}
    \item We define the \textbf{$H$-slope} to be 
    \[
    \mu_H(E)\coloneqq
    \frac{\chi(E)}{H.[E]}, 
    \]
    where $[E] \in H_2(X, \bZ)$ denotes 
    the second homology class of $E$. 
    \item The sheaf $E$ is 
    \textbf{$\mu_H$-semistable} (resp. \textbf{stable}) if 
    for any saturated subsheaf $0 \neq F \subsetneq E$, 
    the inequality 
    \[
    \mu_H(F) \leq \mu_H(E) \quad 
    (\mbox{resp. } \mu_H(F) < \mu_H(E) ) 
    \]
    holds. 
\end{enumerate}
\end{defin}

For a given element 
$v=(\beta, m) \in H_2(X, \bZ) \times \bZ$, 
we denote by $\fM_H(v)$ the moduli stack of 
$\mu_H$-semistable one-dimensional sheaves $E$ satisfying 
\[
[E]=\beta, \quad \chi(E)=m. 
\]

The stack $\fM_H(v)$ admits 
the good moduli space 
$p \colon \fM_H(v) \to M_H(v)$, 
and we have the Hilbert--Chow morphism 
\begin{equation} \label{eq:HC}
\pi_M \colon M_H(v)^{\red} \to \Chow_\beta(X) 
\end{equation}
sending a sheaf $E$ to its fundamental one cycle. 
Here, $\Chow_\beta(X)$ denotes the Chow variety 
of compactly supported effective one cycles 
with homology class $\beta$ 
(see \cite{kol96} for the definition. 
Note that it is denoted as $\Chow'(X)$ in \cite{kol96}). 
We denote by $\pi_\fM$ the composition 
\[
\pi_\fM \colon \fM_H(v)^{\red} \to M_H(v)^{\red} 
\to \Chow_\beta(X). 
\]

Recall from Example \ref{ex:CYderived} that 
the stack $\fM_H(v)$ is the classical truncation 
of a $(-1)$-shifted derived Artin stack. 
In particular, the stack $\fM_H(v)$ carries 
a natural d-critical structure and \eqref{eq:cancot} implies that there exists a natural isomorphism
\[
K^{\vir}_{\fM_H(v)} \cong
\det(\dR p_{\fM*}\dR\mcH om
(\mcE, \mcE ) )|_{\fM_H(v)^{\red}}, 
\]
where $p_\fM \colon \fM_H(v) \times X \to \fM_H(v)$ 
denotes the projection 
and $\mcE$ denotes the universal sheaf on 
$\fM_H(v) \times X$. 
In order to define the well-defined notion of 
Gopakumar--Vafa invariants, 
Maulik--Toda \cite{MT18} and Toda \cite{Tod17} 
proposed the following conjecture 
on the virtual canonical bundle 
of the stack $\fM_H(v)$. 

\begin{conj}[{\cite[Conjecture 2.10]{Tod17}}]
\label{conj:CYcondi}
The stack $\fM_H(v)$ is Calabi--Yau (CY) 
at any point $\gamma \in \Chow_\beta(X)$, i.e., 
there exists an analytic open neighborhood 
$\gamma \in U \subset \Chow_\beta(X)$ 
such that the virtual canonical bundle 
$K^{\vir}_{\fM}$ is trivial on $\pi^{-1}_\fM(U)$. 
\end{conj}

Suppose that Conjecture \ref{conj:CYcondi} holds. 
Then we can take 
an orientation of $\pi^{-1}_\fM(U)$ with 
\[
\left(K^{\vir}_\fM|_{\pi_\fM^{-1}(U)} \right)^{1/2} 
\cong \mcO_{\pi^{-1}_\fM(U)}, 
\]
which we call a \textbf{Calabi--Yau (CY) orientation}. 
As we have seen in \S \ref{ssec:van}, 
we have the associated perverse sheaf 
\[
\varphi_{\fM_H(v)|_U} 
    \in \Perv(\fM_H(v)|_U ). 
\]
We then define the perverse sheaf on the good moduli space as 
\begin{equation} \label{eq:bps}
\varphi_{M_H(v)|_U}
\coloneqq {\mcH^1}( p_*\varphi_{\fM_H(v)|_U}) 
\in \Perv(M_H(v)|_U ), 
\end{equation}
where we denote by 
$\fM_H(v)|_U, M_H(v)|_U$ 
the pull-back of $\fM_H(v), M_H(v)$ 
along the open embedding $U \subset \Chow_\beta(X)$, respectively. 
Note that we denote by $\mcH^i(-)$ 
the $i$-th {\it perverse} cohomology. 

\begin{defin}
Suppose Conjecture \ref{conj:CYcondi} holds. 
For an element $\gamma \in \Chow_\beta(X)$, 
we define a Laurent polynomial $\Phi_H(\gamma, m)$ as follows: 
\begin{equation} \label{eq:GVpoly}
    \Phi_H(\gamma, m)\coloneqq
    \sum_{i \in \bZ}\chi( 
        \mcH^i(\pi_{M*}\varphi_{M_H(v)|_U} ) 
        ) y^i 
        \in \bZ[y^{\pm 1}], 
\end{equation}
where the perverse sheaf 
$\varphi_{M_H(v)|_U}$ is defined as in \eqref{eq:bps}. 
\end{defin}

\begin{rmk}
\begin{enumerate}
\item By \cite[Lemma 2.14]{Tod17}, 
the Laurent polynomial \eqref{eq:GVpoly} is independent 
of the choice of a CY orientation on $\fM_H(v)|_U$. 

\item The definition of the perverse sheaf in \eqref{eq:bps}
is motivated from the notion of BPS sheaves 
for quivers with super-potentials 
introduced by Davison--Meinhardt \cite{DM20}. 
See \cite[Section 2.8]{Tod17} for the detailed discussion. 
\end{enumerate}
\end{rmk}

The following \textbf{$\chi$-independence conjecture} 
is the main subject in this paper: 
\begin{conj}[{\cite[Conjecture 2.15]{Tod17}}] 
\label{conj:chi-indep}
The Laurent polynomial \eqref{eq:GVpoly}
is independent of $m \in \bZ$. 
\end{conj}

At this moment, the above conjecture is known to hold 
in the following cases: 
\begin{itemize}
    \item $X=\Tot_S(\omega_S)$, 
    where $S$ is a smooth projective surface, 
    and $\gamma$ is primitive \cite{Tod17}. 
    \item $X=\Tot_S(\omega_S)$, 
    where $S$ is a del Pezzo surface, 
    and $\gamma$ is arbitrary \cite{MS20, Yua21}. 
    \item $X=\Tot_C(\mcO(D) \oplus \omega_C(-D))$, 
    where $C$ is a smooth projective curve and 
    $D$ is a divisor with $\deg(D) > 2g(C)-2$, 
    and $\gamma$ is arbitrary \cite{MS20}. 
\end{itemize}

\begin{rmk}
Suppose that Conjecture \ref{conj:chi-indep} holds. 
Then we may write the Laurent polynomial \eqref{eq:GVpoly} 
as $\Phi(\gamma)\coloneqq\Phi_H(\gamma, 1)=
\Phi_H(\gamma, m)$ 
for $m \in \bZ$. 
Note that we can drop the subscript $H$ in the notation 
since for $m=1$, the moduli space is independent of 
the choice of an ample divisor $H$. 

Furthermore, for $m=1$, 
we know that the perverse sheaf $\phi_{M_H(v)}$ 
is Verdier self-dual. 
Hence there exist integers $n_{g, \gamma} \in \bZ$ 
for $g \geq 0$ such that the equation
\[
\Phi(\gamma)=\sum_{g \geq 0} n_{g, \gamma}\left(
y^{\frac{1}{2}}+y^{-\frac{1}{2}}
\right)^{2g} 
\]
holds. 
Following Maulik--Toda \cite{MT18}, 
we call the integers $n_{g, \gamma}$ as 
the \textbf{GV invariants} of $X$. 
\end{rmk}

\subsection{Local curves and twisted Higgs bundles}
\label{sec:local}
In this section, we introduce 
a class of Calabi--Yau threefolds which we call local curves. 
Then we review the results 
on the twisted Higgs bundles 
due to Maulik--Shen \cite{MS20}.

\subsubsection{Spectral correspondence for local curves}
Let $C$ be a smooth projective curve
and $N$ be a rank two vector bundle on $C$ with 
$\det N \cong \omega_C$. 
Then the total space $X\coloneqq\Tot_C(N)$ of the bundle $N$ 
gives an example of quasi-projective Calabi--Yau threefolds, 
which we call a \textbf{local curve}. 
Denote by $p \colon X \to C$ the projection. 

In this section, we recall the spectral-type correspondence 
for coherent sheaves on local curves. 
See e.g. \cite{Sim94} for the details. 

\begin{lem} \label{lem:spec}
Giving a compactly supported 
pure one-dimensional coherent sheaf on $X$ 
is equivalent to giving a pair $(E, \phi)$ 
of a locally free sheaf $E$ on $C$ 
and a morphism $\phi \in \Hom(E, E \otimes N)$ 
satisfying $\phi \wedge \phi =0$. 
\end{lem}

We call a pair $(E, \phi)$ in the above lemma as 
an \textbf{$N$-Higgs bundle}. 
We can define the slope semistability for $N$-Higgs bundles 
as in Definition \ref{def:slope}: 

\begin{defin} \label{def:slopeHiggs}
Let $(E, \phi)$ be an $N$-Higgs bundle. 
\begin{enumerate}
    \item We define the \textbf{slope} of $(E, \phi)$ as 
    \[
    \mu(E)\coloneqq\frac{\chi(E)}{\rk(E)}. 
    \]
    \item The $N$-Higgs bundle $(E, \phi)$ is 
    \textbf{$\mu$-semistable} (resp. \textbf{stable}) if 
    for any saturated subsheaf $0 \neq F \subsetneq E$ 
    with $\phi(F) \subset F \otimes N$, 
    the inequality 
    \[
    \mu(F) \leq \mu(E) \quad 
    (\mbox{resp. } \mu(F) < \mu(E) ) 
    \]
    holds. 
\end{enumerate}
\end{defin}

\begin{lem} \label{lem:spec-stab}
Take an ample divisor $H$ on $C$. 
Let $\mcE$ be a pure one-dimensional coherent sheaf on $X$ and
$(E, \phi)$ be the corresponding $N$-Higgs bundle. 

Then the sheaf $\mcE$ is 
$\mu_{p^*H}$-(semi)stable 
if and only if 
the $N$-Higgs bundle $(E, \phi)$ is 
$\mu$-(semi)stable. 
\end{lem}

Let $\fM^{\semist}_X(r, m)\coloneqq\fM_{p^*H}(r[C], m)$ 
be the moduli stack of $\mu_{p^*H}$-semistable 
sheaves $\mcE$ on $X$ satisfying 
$[\mcE]=r[C]$ and $\chi(\mcE)=m$. 
Let $M^{\semist}_X(r, m)$ be the good moduli space of 
$\fM^{\semist}_X(r, m)$. 
By the above lemma, 
$\bC$-valued points of $\fM^{\ss}_X(r, m)$ correspond to $\mu$-semistable $N$-Higgs bundles. 

The moduli space $M^{\semist}_X(r, m)$ admits 
a Hitchin type morphism: 
Define a \textbf{Hitchin base} as 
\[
B_X\coloneqq\bigoplus_{i=1}^r \mH^0(C, \Sym^i(N) ), 
\]
and a \textbf{Hitchin morphism} as follows: 
\begin{equation} \label{eq:HitX}
h_X \colon M^{\semist}_X(r, m) \to B_X, \quad 
(E, \phi) \mapsto 
( 
\tr(\phi^i)
)_{i=1}^r, 
\end{equation}
where $\phi^i \colon E \to E \otimes \Sym^i(N)$ 
is obtained by the $i$-th iteration of $\phi$. 

\begin{rmk} \label{rmk:hitchow}
We can construct a bijection between 
the sets of closed points of 
$\im(h_X)$ and $\im(\pi_M)$ 
by sending a point in $\im(h_X)$ 
to its spectral curve, where 
$\pi_M \colon M^{\semist}_X(r, m)^{\red} \to \Chow_{r[C]}(X)$ denotes the Hilbert--Chow morphism
defined as in \eqref{eq:HC}. 
Moreover, by the properness of the morphisms 
$h_X$ and $\pi_M$, the spaces 
$\im(h_X)$ and $\im(\pi_M)$ are homeomorphic. 

As a result, the GV invariants 
do not change if we replace 
the Hilbert--Chow morphism with 
the Hitchin morphism. 
Hence we use the Hitchin morphism 
for the GV theory of local curves in this paper. 
\end{rmk}

\subsubsection{Twisted Higgs bundles}
Let $L$ be a line bundle on a smooth projective curve $C$. 
Denote by $Y \coloneqq \Tot_C(L)$ the total space of $L$. 

An \textbf{$L$-Higgs bundle} is a pair $(E, \theta)$ 
consisting of a locally free sheaf $E$ on $C$ 
and a homomorphism $\theta \in \Hom(E, E \otimes L)$. 
For the canonical divisor $L=K_C$, 
the notion of $K_C$-Higgs bundles 
agrees with the usual notion of Higgs bundles. 

As in Definition \ref{def:slopeHiggs}, 
we can define the notion of $\mu$-semistability 
for $L$-Higgs bundles. 
We denote by $\fM^{\semist}_Y(r, m)$ 
the moduli stack of $\mu$-semistable 
$L$-Higgs bundles $(E, \theta)$ with 
$\rk(E)=r, \chi(E)=m$, 
and $M^{\semist}_Y(r, m)$ its good moduli space. 
Similarly to \eqref{eq:HitX}, 
we have a Hitchin morphism 
\begin{equation} \label{eq:HitS}
h_Y \colon M^{\semist}_Y(r, m) \to B_Y \coloneqq \bigoplus_{i=1}^r \mH^0(C, L^{\otimes i})
\end{equation}
sending an $L$-Higgs bundle $(E, \theta)$ 
to $(\tr(\theta^i))_{i=1}^r$. 

We denote by 
$\widetilde{h}_Y \colon \fM^{\semist}_Y(r, m) \to B_Y$ 
the composition 
\begin{equation} \label{eq:HitSstack}
    \widetilde{h}_Y \colon \fM^{\semist}_Y(r, m) 
    \to M^{\semist}_Y(r, m) \to B_Y.
\end{equation}

Given an element $a \in B_Y$, 
we denote by $C_a \subset Y$ 
its spectral curve. 
Define an open dense subset $U \subset B_Y$ as 
\[
U\coloneqq\left\{
    a \in B_Y : C_a \mbox{ is smooth}
    \right\}, 
\]
and let $g \colon \mcC \to U$ be the universal spectral curve. 
The following result plays a key role in this paper: 

\begin{thm}[{\cite[Theorem 0.4]{MS20}}]
\label{thm:MS}
Suppose that $\deg(L) > 2g(C)-2$. 
Then we have an isomorphism 
\[
 h_{Y*}\sIC_{M^{\semist}_Y(r, m)} \cong 
\bigoplus_{i=0}^{2d}\sIC(
\wedge^i \dR^1 g_*\bQ_\mcC 
)[-i+d], 
\]
where $d$ denotes the genus of the fibers of $g \colon \mcC \to U$. 

In particular, we have isomorphisms 
\[
h_{Y*}\sIC_{M^{\semist}_Y(r, m)} 
\cong  h_{Y*}\sIC_{M^{\semist}_Y(r, m')}
\]
for all $m, m' \in \bZ$. 
\end{thm}

%

\section{Cohomological \texorpdfstring{$\chi$}{chi}-independence for local curves}\label{sec:chi}
Let $C$ be a smooth projective curve of genus $g$ and
$N$ be a rank two vector bundle on $C$ 
with $\det N \cong \omega_C$. 
We put $X\coloneqq\Tot_C(N)$. 
The goal of this section is to prove the following theorem: 

\begin{thm} \label{thm:chi-lC}
Let $X=\Tot_C(N)$ be a local curve. 
For every positive integer $r \in \bZ_{>0}$ 
and a class $\gamma \in B_X$, 
Conjecture \ref{conj:chi-indep} holds. 
\end{thm}

\subsection{Global d-critical charts for moduli spaces on local curves}

We first recall the main result of the companion paper \cite{KM21}:

\begin{thm}{{\cite[Theorem 5.6, Proposition 5.7]{KM21}}}\label{thm:KM}
Let $C$ be a smooth projective curve and take a short exact sequence 
\begin{equation}\label{eq:short}
0 \to L_1 \to N \to L_2 \to 0
\end{equation}
of locally free sheaves on $C$ where $L_1$ and $L_2$ are rank one. 
Suppose that there exists an isomorphism $\det(N) \cong \omega_C$, and the inequality $\deg(L_2) > 2 g(C) - 2 $ holds.
Write $X =\Tot_C(N)$ and $Y = \Tot_{C}(L_2)$.
Let $\bs{\fM}_{X}$ and $\bs{\fM}_Y$ be the derived moduli stack of compactly supported coherent sheaves on $X$ and $Y$ respectively.
\begin{enumerate}
    \item[(i)] 
There exists a function $f$ on $\bs{\fM}_Y$ such that the projection from $X$ to $Y$ induces an equivalence of $(-1)$-shifted symplectic derived Artin stacks
\begin{equation}\label{eq:KM}
\bs{\fM}_X \simeq \DCrit(f).
\end{equation}
\item[(ii)]

Let $(E, \phi)$ be an $L_2$-Higgs bundle.
Then we have an equality
\[
f([(L_2, \phi)]) = 1/2 \cdot \alpha(\tr(\phi^2))
\]
where $\alpha \in 
\mH^0(C, L_2^{\otimes 2})^\vee 
\cong \Ext^1(L_2, L_1)$ is the class corresponding to the short exact sequence \eqref{eq:short}.

\end{enumerate}
\end{thm}

We now want to describe the moduli stack of \textit{semistable} $N$-Higgs bundle as a
global critical locus.
We begin with the following easy lemma: 
\begin{lem} \label{lem:example}
Let $C$ be a smooth projective curve and $N$ be a rank two vector bundle on $C$.
Then we can take the short exact sequence \eqref{eq:short} so that $\deg(L_2) > 2 g(C) - 2 $ holds. More generally, we can take $L_2$ so that its degree is arbitrarily large.
\end{lem}
\begin{proof}
Let $\mcO_C(1)$ be an ample line bundle on $C$. 
Then there exists an integer $l_0 >0$ such that 
for every integer $l \geq l_0$, 
the bundle $N^\vee(l)$ is globally generated. 
Then a general element 
$s \in \Hom(N, \mcO_C(l)) \cong \mH^0(C, N^\vee(l))$ 
is surjective. 
Putting $L_2\coloneqq\mcO_C(l)$ and $L_1\coloneqq\Ker(s)$, 
we get the desired exact sequence as in \eqref{eq:short}. 
\end{proof}

\begin{lem} \label{lem:stability}
Take integers $r, m \in \bZ$ with $r > 0$. 
Then there exists an integer $k(r) > 2g(C)-2$ 
depending only on $r$, such that,
for any short exact sequence \eqref{eq:short} with 
$\deg(L_2) \geq k(r)$, 
the following statement holds: 
For every $\mu$-semistable $N$-Higgs bundle 
$(E, \phi) \in \fM^{\semist}_X(r, m)$, 
the $L_2$-Higgs bundle $(E, s \circ \phi)$ is $\mu$-semistable. 
\end{lem}
\begin{proof}
Let $(E, \phi) \in \fM^{\semist}_X(r, m)$ 
be a $\mu$-semistable $N$-Higgs bundle. 
Suppose that the $L_2$-Higgs bundle 
$(E, s \circ \phi)$ is not $\mu$-semistable. 
We claim that there exists a saturated subsheaf 
$F \subset E$ such that 
$\mu(F) > \mu(E)$ and 
$\Hom(F, E/F \otimes L_1) \neq 0$. 
Indeed, let $F \subset E$ be 
the maximal destabilizing subsheaf 
of $(E, s \circ \phi)$. 
This means that 
we have 
$\mu(F) > \mu(E)$ and 
$(s \circ \phi)(F) \subset F \otimes L_2$. 
The latter condition is equivalent that 
the composition
\[
F \hookrightarrow E \xrightarrow{s\circ\phi} E \otimes L_2 
\to E/F \otimes L_2
\]
is zero. 
On the other hand, by 
the $\mu$-semistability of $(E, \phi)$, 
we have 
$\phi(F) \nsubseteq F \otimes N$, i.e., 
the composition 
\[
F \hookrightarrow E \xrightarrow{\phi} E \otimes N 
\to E/F \otimes N 
\]
is non-zero. 
As a result, we obtain the following diagram 
\[
\xymatrix{
& &F \ar[d]^{\neq 0} \ar[rd]^0 \ar@{.>}[ld] & \\
&E/F \otimes L_1 \ar[r] &E/F \otimes N \ar[r] &E/F \otimes L_2, 
}
\]
hence we have $\Hom(F, E/F \otimes L_1) \neq 0$. 

By Lemma \ref{lem:bounded} below, 
we can replace an exact sequence \eqref{eq:short} 
so that $\Hom(F, E/F \otimes L_1)=0$ 
for all $\mu$-semistable $N$-Higgs bundles 
$(E,\phi) \in \fM^{\semist}_X(r, m)$ 
and all saturated subsheaves $F \subset E$ 
with $\mu(F) > \mu(E)$. 
Hence the above argument shows that 
$(E, s \circ \phi)$ remains $\mu$-semistable 
for such a choice of the exact sequence \eqref{eq:short}. 
\end{proof}

\begin{lem} \label{lem:bounded}
Take integers $r, m \in \bZ$ with $r >0$. 
Then the following sets are bounded: 
\begin{align*}
    &\mcS\coloneqq\left\{
        F, E/F : 
        \begin{aligned}
        &(E, \phi) \in \fM^{\semist}_X(r, m), \\
        &F \subset E \mbox{ is saturated with } 
        \mu(F) > \mu(E)
        \end{aligned}
        \right\}, \\
    &\HN(\mcS)\coloneqq\left\{
        A : A \mbox{ is a Harder--Narasimhan factor of } G \in \mcS
        \right\}. 
\end{align*}

Moreover, there exists an integer $k'(r) \in \bZ$, 
depending only on $r$, such that 
for all line bundles $L_1$ with $\deg(L_1) \leq k'(r)$ 
and for all $F, E/F \in \mcS$, 
we have the vanishing $\Hom(F, E/F \otimes L_1)=0$. 
\end{lem}
\begin{proof}
The boundedness of the sets $\mcS, \HN(\mcS)$ 
follows from the boundedness of $\fM^{\semist}_X(r, m)$ 
and Grothendieck's boundedness lemma 
(cf. \cite[Lemma 1.7.9]{HL97}). 

In particular, 
there exist integers $a, b$ such that 
the inequalities 
$\mu_{\min}(F) \geq a$ and 
$\mu_{\max}(F/E) \leq b$ 
hold for all $F, E/F \in \mcS$. 
By setting $k'(r, m)\coloneqq a-b-1$, 
we obtain the inequality 
\[
\mu_{\min}(F) > \mu_{\max}(E/F \otimes L_1)
\]
for all line bundles $L_1$ with $\deg(L_1) \leq k'(r, m)$ 
and all $F, E/F \in \mcS$. 

Finally, observe that we have an isomorphism 
$\fM^{\semist}_X(r, m) \cong \fM^{\semist}_X(r, m+r)$ 
by tensoring with a degree one line bundle. 
Hence by putting 
$k'(r) \coloneqq \min\{k'(r, m) \colon m=0, \ldots, r-1\}$, 
the second assertion holds. 
\end{proof}

\begin{prop} \label{prop:dchart}
Let $r, m \in \bZ$ be integers with $r>0$. 
Take an exact sequence \eqref{eq:short} 
as in Lemma \ref{lem:stability}. 
Let 
$\alpha \in 
\mH^0(C, L_2^{\otimes 2})^\vee 
\cong \Ext^1(L_2, L_1)$ 
be the corresponding class. 
Denote by $Y \coloneqq \Tot_C(L_2)$. 
Define the function $g \colon B_Y \to \bA^1$ as 
\begin{equation} \label{eq:fctBS}
g \colon B_Y=\bigoplus_{i=1}^r \mH^0(C, L_2^{\otimes i})
\to \mathbb{A}^1, \quad 
(a_i)_{i=1}^r \mapsto 1/2 \cdot \alpha(a_2). 
\end{equation}

Then we have an isomorphism 
\begin{equation} \label{eq:dchart}
    \fM^{\semist}_X(r, m) \cong 
    \left\{d\left(g \circ \widetilde{h}_Y \right)=0 \right\} 
    \subset \fM^{\semist}_Y(r, m)
\end{equation}
which preserves the d-critical structure.
\end{prop}
\begin{proof}
By Lemma \ref{lem:stability}, 
the isomorphism \eqref{eq:KM}
restricts to the semistable loci. 
Then the claim
follows from Theorem \ref{thm:KM} (ii), the fact that the derived moduli stack $\bs{\fM}^{\semist}_Y(r, m)$ is a smooth (classical) stack, and that the classical truncation of the derived critical locus of a function on a smooth stack coincides with the classical critical locus.
\end{proof}

For the vanishing cycle sheaves on the good moduli spaces, 
we have the following result: 
\begin{prop} \label{prop:vanMS}
Let $r, m, m' \in \bZ$ be integers with $r>0$. 
Let $g \colon B_Y \to \bA^1$ be a function 
as in Proposition \ref{prop:dchart}. 
Then we have an isomorphism 
\begin{equation} \label{eq:vanMS}
     h_{Y*}\left( 
    \varphi_{g \circ h_Y}\left(\sIC_{M^{\semist}_Y(r, m)} \right) 
    \right)
    \cong 
     h_{Y*}\left( 
    \varphi_{g \circ h_Y}\left(\sIC_{M^{\semist}_Y(r, m')} \right) 
    \right), 
\end{equation}
where $h_Y \colon M^{\semist}_Y(r, m) \to B_Y$, 
$h_Y \colon M^{\semist}_Y(r, m') \to B_Y$ 
denote the Hitchin morphisms \eqref{eq:HitS}
on the good moduli spaces. 
\end{prop}
\begin{proof}
Since the Hitchin morphism 
$h_Y \colon M^{\semist}_Y(r, m) \to B_Y$ 
is proper, the result follows from 
Theorem \ref{thm:MS} together with 
the commutativity of the vanishing cycle functors 
and proper push forwards. 
\end{proof}

In the following subsections, 
we will show that the complexes in \eqref{eq:vanMS}
compute the generalized GV invariants 
for the local curve $X=\Tot_C(N)$.

\subsection{CY property for local curves}
In this subsection, 
we fix integers $r, m \in \bZ$ with $r>0$, 
and an exact sequence \eqref{eq:short}. 
We assume that the line bundle $L_2$ satisfies 
the following conditions: 
\begin{itemize}
    \item We have $\deg(L_2) \geq k(m)$ 
    (see Lemma \ref{lem:stability}), 
    \item $L_2$ is globally generated. 
\end{itemize}
Recall that we denote as 
$X \coloneqq \Tot_C(N)$ and $Y \coloneqq \Tot_C(L_2)$. 
By Proposition \ref{prop:dchart}, 
the moduli stack $\fM^{\semist}_X(r, m)$ 
is written as the global critical locus 
inside $\fM^{\semist}_Y(r, m)$. 

\begin{prop} \label{prop:CYori}
The canonical bundle $K_{\fM^{\semist}_Y(r, m)}$ 
of the stack $\fM^{\semist}_Y(r, m)$ is trivial,
and hence so is the virtual canonical bundle 
$K^{\vir}_{\fM^{\semist}_X(r, m)}$ 
of the stack $\fM^{\semist}_X(r, m)$. 
In particular, 
the stack $\fM^{\semist}_X(r, m)$ is CY at any point 
$\gamma \in B_X$, i.e., 
Conjecture \ref{conj:CYcondi} holds for 
$\fM^{\semist}_X(r, m)$. 
\end{prop}
\begin{proof}
A similar argument can be found in 
\cite[Theorem 7.1]{Tod17}. 
Take a morphism $T \to \fM^{\semist}_Y(r, m)$ 
from a scheme $T$. 
Let $\mcE \in \Coh(Y \times T)$ be 
the corresponding family of $\mu$-semistable 
one-dimensional sheaves on $Y$. 
We consider the following diagram: 
\[
\xymatrix{
Y_T \coloneqq Y \times T \ar[r]^-{\pi_T} \ar[d]_{p_T} & T \\
C_T \coloneqq C \times T. \ar[ru]_{q_T} & 
}
\]
We need to construct an isomorphism 
\[
\det\dR\mcH om_{\pi_T}(\mcE, \mcE) \cong \mcO_T, 
\]
which is functorial in $T$. 
We have the following exact sequence 
\begin{equation} \label{eq:expT}
    0 \to p_T^*(L_2^{-1} \boxtimes p_{T*}\mcE ) 
    \to p_T^*p_{T*}\mcE \to \mcE \to 0. 
\end{equation}
Applying the functor 
$\dR\mcH om_{\pi_T}(-, \mcE)$ to 
the exact sequence \eqref{eq:expT}, we obtain 
the exact triangle 
\begin{equation*}
    \dR\mcH om_{\pi_T}(\mcE, \mcE) \to 
    \dR\mcH om_{q_T}(\mcF, \mcF) \to 
    \dR\mcH om_{q_T}(\mcF \boxtimes L_2^{-1}, \mcF), 
\end{equation*}
where we put $\mcF\coloneqq p_{T*}\mcE$. 
By taking the determinants, we get 
\begin{equation} \label{eq:detE}
    \det \dR\mcH om_{\pi_T}(\mcE, \mcE) \cong 
    \det  \dR\mcH om_{q_T}(\mcF, \mcF) \otimes 
    ( \det \dR\mcH om_{q_T}(\mcF \boxtimes L_2^{-1}, \mcF) )^{-1}. 
\end{equation}

On the other hand, 
we have an exact sequence 
\[
0 \to \mcO_C \to L_2 \to \mcO_Z \to 0, 
\]
where $Z \in |L_2|$ is a finite set of points. 
Applying the functor 
$\dR\mcH om_{q_T}(\mcF, \mcF \boxtimes (-))$ 
and taking the determinants, 
we get 
\begin{equation} \label{eq:detF}
\begin{aligned}
    \det \dR\mcH om_{q_T}(\mcF \boxtimes L_2^{-1}, \mcF) 
    &\cong \det \dR\mcH om_{q_T}(\mcF, \mcF \boxtimes L_2) \\
    &\cong \det \dR\mcH om_{q_T}(\mcF, \mcF) \otimes 
    \det \dR\mcH om_{q_T}(\mcF, \mcF_Z), 
\end{aligned}
\end{equation}
where we put $\mcF_Z\coloneqq\mcF|_{Z \times T}$. 
Combining the equations \eqref{eq:detE} and \eqref{eq:detF}, 
we obtain the desired isomorphism 
\begin{equation*}
\begin{aligned}
    \det\dR\mcH om_{\pi_T}(\mcE, \mcE) 
    &\cong \det \dR\mcH om_{q_T}(\mcF, \mcF_Z)^{-1} \\
    &\cong \det \dR\mcH om_{r_T}(\mcF_Z, \mcF_Z)^{-1} \\
    &\cong \bigotimes_{i=1}^k 
    \det\dR\mcH om(\mcF_{p_i}, \mcF_{p_i}) \cong \mcO_T, 
\end{aligned}
\end{equation*}
where we denote by $r_T \colon Z \times T \to T$ the projection. 
For the third isomorphism, 
we put $Z=\{p_1, \ldots, p_k \}$ 
and $\mcF_{p_i}\coloneqq\mcF|_{\{p_i\} \times T}$. 

The triviality of the virtual canonical bundle 
$K^{\vir}_{\fM^{\semist}_X(r, m)}$ 
now follows from Lemma \ref{lem:natori}. 
\end{proof}


\subsection{Proof of Theorem \ref{thm:chi-lC}}
In this subsection, we finish 
the proof of Theorem \ref{thm:chi-lC}. 

Consider the following commutative diagram: 
\begin{equation} \label{eq:Hitcomm}
\xymatrix{
&M^{\semist}_X(r, m) \ar@{^{(}->}[r]^\iota 
\ar[d]_{h_X} &M^{\semist}_Y(r, m) \ar[d]^{h_Y} \\
&B_X \ar[r]_b &B_Y, 
}
\end{equation}
where the morphism $b \colon B_X \to B_Y$ 
is induced by the surjection 
$\Sym^k(N) \to L_2^{\otimes k}$ 
for $k=1, \ldots, r$. 

\begin{lem} \label{lem:Hitfin}
The morphism 
\begin{equation} \label{eq:Hitfin}
b|_{\im(h_X)} \colon 
\im(h_X) \to B_Y,
\end{equation}
is finite. 
\end{lem}
\begin{proof}
It is enough to show that the morphism 
in \eqref{eq:Hitfin} is proper and affine. 
The composition 
$h_Y \circ \iota = b \circ h_X \colon 
M^{\semist}_X(r, m) \to B_Y$ 
is proper as so are $h_Y$ and $\iota$. 
Furthermore, the morphism 
$h_X \colon M^{\semist}_X(r, m) \to \im(h_X)$ 
is proper and surjective. 
Hence the morphism \eqref{eq:Hitfin} is proper. 

On the other hand, by the properness of 
the Hitchin morphism $h_X$, 
the inclusion $\im(h_X) \hookrightarrow B_X$ is closed. 
As the morphism $b \colon B_X \to B_Y$ is just 
the projection of affine spaces, it is also affine. 
We conclude that the composition 
\[
\im(h_X) \hookrightarrow B_X \xrightarrow{b} B_Y
\]
is affine, as required. 
\end{proof}

Recall from \eqref{eq:HitS} and \eqref{eq:HitSstack} 
that we denote by 
$h_Y \colon M^{\semist}_Y(r, m) \to B_Y$, 
$\widetilde{h}_Y \colon \fM^{\semist}_Y(r, m) \to B_Y$ 
the Hitchin morphisms. 
Recall also that we have the function 
$g \colon B_Y \to \bA^1$ 
defined in Proposition \ref{prop:dchart}. 
We equip $\fM^{\semist}_X(r, m)$ with the orientation defined by the global critical chart description in Proposition \ref{prop:dchart} and Lemma \ref{lem:natori}.
We define the vanishing cycle complex $\varphi_{\fM^{\semist}_X(r, m)} \in \Perv(\fM^{\semist}_X(r, m))$ using this orientation. We set
\begin{equation}\label{eq:defvangood}
\varphi_{M^{\semist}_X(r, m)} \coloneqq \cH^1(p_{X*}\varphi_{\fM^{\semist}_X(r, m)}) \in \Perv(M^{\semist}_X(r, m))
\end{equation}

We need the following proposition: 
\begin{prop}\label{prop:isomvan} \label{prop:BPSvsVan}
We have isomorphisms 
\[
\varphi_{M^{\semist}_X(r, m)} \cong 
{\mcH^1}(p_{X*}(
\varphi_{g \circ \widetilde{h}_Y}(
\sIC_{\fM^{\semist}_Y(r, m)} 
)))
\cong 
\varphi_{g \circ h_Y}( 
\sIC_{M^{\semist}_Y(r, m)}
). 
\]
\end{prop}
We postpone the proof of this proposition 
until the next section.

\begin{proof}[Proof of Theorem \ref{thm:chi-lC}]
Let us take integers $r, m, m' \in \bZ$ 
with $r > 0$. 
Let $k(r) \in \bZ$ be integers 
as in Lemma \ref{lem:stability}. 
We take an exact sequence \eqref{eq:short} 
such that $L_2$ is globally generated and 
$\deg(L_2) \geq k(r)$ holds. 
Then by Proposition \ref{prop:dchart}, 
the moduli stacks 
$\fM^{\semist}_X(r, m), \fM^{\semist}_X(r, m')$ 
are written as the global critical loci 
inside the stacks 
$\fM^{\semist}_Y(r, m), \fM^{\semist}_Y(r, m')$, 
respectively. 

Recall from Proposition \ref{prop:CYori} 
that the canonical bundles 
$K_{\fM^{\semist}_Y(r, m)}$ 
and $K_{\fM^{\semist}_Y(r, m')}$ 
are trivial, 
hence the natural orientation data 
in Lemma \ref{lem:natori} is 
a CY orientation data.

Let $\varphi_{M^{\semist}_X(r, m)}, \varphi_{M^{\semist}_X(r, m')}$ 
be the associated perverse sheaves on $M^{\semist}_X(r, m)$, $M^{\semist}_X(r, m')$, 
defined as in \eqref{eq:defvangood}. 
We have an isomorphism 
\begin{equation} \label{eq:isoBS}
 h_{Y*}(\varphi_{M^{\semist}_X(r, m)} ) 
\cong 
 h_{Y*}(\varphi_{M^{\semist}_X(r, m')}) 
\end{equation}
by Proposition \ref{prop:vanMS}. 
By using the commutative diagram \eqref{eq:Hitcomm}, 
we can rewrite the left hand side of \eqref{eq:isoBS} 
as 
\begin{equation*} \label{eq:pushcomm}
 h_{Y*}(\varphi_{M^{\semist}_X(r, m)} ) 
\cong b_* h_{X*}(\varphi_{M^{\semist}_X(r, m)}). 
\end{equation*}
By Lemma \ref{lem:Hitfin}, 
the map (\ref{eq:Hitfin}) is finite. 
Since the push-forward along a finite morphism preserves 
the perverse t-structures, 
we obtain 
\begin{equation} \label{eq:p-coh}
    {\mcH^i}( 
     h_{Y*}(\varphi_{M^{\semist}_X(r, m)} ) 
    )
    \cong b_*{\mcH^i}( 
     h_{X*}(\varphi_{M^{\semist}_X(r, m)}
    )) 
\end{equation}
for all $i \in \bZ$, 
and we have the same isomorphisms 
if we replace the integer $m$ with $m'$. 

Combining the isomorphisms 
\eqref{eq:isoBS} and \eqref{eq:p-coh}, 
and taking the Euler characteristics, 
we conclude that 
\[
\chi({\mcH^i}( 
     h_{X*}(\varphi_{M^{\semist}_X(r, m)}
    )))
=\chi({\mcH^i}( 
     h_{X*}(\varphi_{M^{\semist}_X(r, m')}
    )) )
\]
as desired. 
\end{proof}

\section{Cohomological integrality theorem 
for twisted Higgs bundles} \label{sec:int}

In this section, we prove the cohomological integrality theorem in the sense of \cite[\S 1.3]{DM20} for twisted Higgs bundles. 
Since semistable twisted Higgs bundles form 
a category of homological dimension one, 
we can prove the cohomological integrality theorem 
using the techniques of \cite{DM20, Mei15}, 
which treat the case of quivers. 

\subsection{Mixed Hodge modules on stacks}\label{ssec:MHM}

Here we give a quick introduction to mixed Hodge modules, which is a sheaf theoretic version of mixed Hodge structures introduced by Morihiko Saito \cite{Sai90}. An advantage of working with mixed Hodge modules (rather than perverse sheaves) is the fact that the category of pure Hodge modules is semi-simple. In particular, an equality of Grothendieck group of the category of pure Hodge modules implies an isomorphism between them.
This was used by Davison--Meinhardt \cite{DM20} in their proof of the cohomological integrality theorem for quivers with potentials, and will be used in the proof of Theorem \ref{thm:int-IC}.

Let $X$ be a separated scheme locally of finite type over complex number whose connected components are quasi-compact.
For such $X$, we can define the category of \textbf{mixed Hodge modules} $\MHM(X)$ and its bounded derived category $D^b(\MHM(X))$ which admits a six-functor formalism (see \cite{Sai89} for an overview).
There exists an exact functor 
\[
\rat \colon D^b(\MHM(X)) \to D^b(\Perv(X))
\]
which restricts to a faithful functor $\MHM(X) \to \Perv(X)$.
The functor $\rat$ is compatible with all six functors.
A mixed Hodge module $M$ is equipped with an increasing filtration called the \textbf{weight filtration} which we denote by $W_{\bullet}M$.
A mixed Hodge module $M \in \MHM(X)$ is called \textbf{pure of weight $i$} if $W_{i-1}M = 0$ and $W_i M = M$ holds, and an object $M^{\bullet} \in D^b(\MHM(X))$ is called pure if the $i$-th cohomology mixed Hodge module $\mcH^i(M)$ is pure of weight $i$.

The category of mixed Hodge modules over a point is equivalent to the category of graded polarizable mixed Hodge structures.
Let $a_X \colon X \to \Spec \bC$ be the constant map to a point. Then we define objects $\bQ_X^{}, \bD\bQ_X \in D^b(\MHM(X))$ as
\[
\bQ_X \coloneqq a_X^* \bQ, \  \bD \bQ_X \coloneqq a_X^! \bQ,
\]
where $\bQ$ denotes the mixed Hodge structure of weight zero and dimension one.

The category of mixed Hodge modules forms a stack in the smooth topology (see \cite[Theorem 2.3]{Arc}).
This motivates the following definition of the category of mixed Hodge modules on an Artin stack $\fX$.
\begin{defin}
Let $\fX$ be a complex Artin stack.
We let $\Sch^{\sm, \sep}_{/\fX}$ denote the category of separated schemes smooth and of finite type over $\fX$. A \textbf{mixed Hodge module} on $\fX$ is a pair consisting of an assignment 
\[
  \Sch^{\sm, \sep}_{/\fX} \ni (t \colon T \to \fX) \mapsto M_t \in \MHM(T)
\]
and a choice of an isomorphism
\[
\theta_q \colon q^* M_{t_2}[\dim q] \cong M_{t_1}
\]
for each smooth morphism $q \colon (t_1 \colon T_1 \to \fX) \to (t_2 \colon T_2 \to \fX)$ in $\Sch^{\sm, \sep}_{/\fX}$ satisfying the associativity relation. Mixed Hodge modules on $\fX$ form a category $\MHM(\fX)$ in the natural way. We have a natural forgetful functor
\[
\rat \colon \MHM(\fX) \to \Perv(\fX).
\]
\end{defin}

Take a smooth surjective morphism from a separated finite type scheme $t \colon T \to \fX$ and we let $\pr_i \colon T \times_{\fX} T \to T$ denote the $i$-th projection. Then we can identify $\MHM(\fX)$ with the category of pairs $(M, \sigma)$, where $M$ is a mixed Hodge module on $T$ and $\sigma$ is an isomorphism
\[
\sigma \colon \pr_1^* M \cong \pr_2^* M
\]
satisfying the cocycle condition.

At present we do not have a full six-functor formalism for mixed Hodge modules on Artin stacks.
However we have some part of it 
which is sufficient for applications in this paper.
Firstly, if we are given a smooth morphism $q \colon \fX_1 \to \fX_2$ between Artin stacks, we can define a functor
\[
q^*[\dim q] \colon \MHM(\fX_2) \to \MHM(\fX_1)
\]
in the standard way.

Now assume that we are given a finite type morphism $p \colon \fX \to X$ from an Artin stack to a separated finite type scheme.
We want to define the functor 
\[
\cH^n(p_* (-)) \colon \MHM(\fX) \to \MHM(X)
\]
compatible with the functor $\rat$.  Here $\cH^n$ denotes the $n$-th cohomology with respect to the perverse t-structure on $D(\MHM(X))$. 
We assume that the morphism $p$ satisfies the following assumption: 
   \begin{quote}$\mathrm{(*)}$
   \label{quote:star}
         For each object $\cF \in D^b_c(\fX)$ in the bounded derived category of sheaves on $\fX$ with constructible cohomology
         and integer $N$, there exists a smooth morphism from a scheme $q_N \colon T_N \to \fX$ such that the natural map
        \[
    \mcH^n_{}(\cF) \to \mcH^n_{}({q_N}_* q_N^* \cF)
    \]
is isomorphism for each $n \leq N$. Here $\mcH^n$ denotes the perverse t-structure on $D^b_c(\fX)$.
    \end{quote}
    This assumption is automatically satisfied when $\fX$ is of the form $[Y / G]$ for some scheme $Y$ and a linear algebraic group $G$ (see \cite[\S 2.3.2]{Dav21}). 
Let $p \colon \fX \to X$ be a morphism to a scheme.
For a mixed Hodge module $M \in \MHM(\fX)$ and $n \in \bZ$, we define a mixed Hodge module
\[
\mcH^n_{}(p_* M) \coloneqq \mcH^n_{}((p \circ q_N)_* q_N^* M)
\]
where $N$ is a sufficiently large integer.
We can show that $\mcH^n_{}(p_* M)$ is independent of the choice of $N$ and $q_N$.
If we take $p$ as the constant map $a_{\fX} \colon \fX \to \Spec \bC$,
we can construct a mixed Hodge structure $\mH^n(\fX, M) \coloneqq \mcH^n({a_{\fX}}_* M)$.
Similarly, we can extend the perverse sheaves
$\mcH^n_{}(p_* \bQ_{\fX})$ and $\mcH^n(p_* \bD_{\fX})$ to  mixed Hodge modules, and the vector spaces $\mH^n(\fX)$ and $\HBM_{n}(\fX)$ to mixed Hodge structures.

For a complex of mixed Hodge modules $M \in D(\MHM(X))$, we define
\[ 
\cH(M) \coloneqq \bigoplus_{i \in \bN} \cH^i(M)[-i].
\]
\begin{lem}\label{lem:decompthm}
Let $\fX$ be a stack satisfying the condition $\mathrm{(*)}$, $p \colon \fX \to X$ be a morphism to a separated finite type complex scheme, and $h\colon X \to B$ be a proper morphism between separated finite type complex schemes.  Take $M \in D^b(\MHM(X))$ and assume that $\cH(p_* M)$ is pure.
Then we have an isomorphism
\[
\cH(h_* \cH(p_* M)) \cong \cH((h \circ p)_* M). 
\]
\end{lem}

\begin{proof}
Take an integer $n$ and an integer $N$ such that $N > n + \dim h^{-1}(x)$ holds for each $x \in B$. Then \cite[Corollary 5.2.14]{Dim04} implies that we have an isomorphism
\[
\cH^n(h_* \cH(p_* M)) \cong \cH^n(h_* \tau^{\leq N}\cH(p_* M)).
\]
Take a smooth morphism $q \colon T \to X$ such that we have isomorphisms
\begin{align*}
\tau^{\leq N}\cH(p_* M)) &\cong \tau^{\leq N}\cH((p \circ q)_* q^* M)), \\
\cH^n((h \circ p)_*  M) &\cong \cH^n((h \circ p \circ q)_* q^* M).
\end{align*}
Then what it is enough to prove the following isomorphism
\[
\cH^n(h_*\tau^{\leq N}\cH((p \circ q)_* q^* M))) \cong \cH^n((h \circ p \circ q)_* q^* M).
\]
Saito's decomposition theorem implies an isomorphism
\[
\cH^n(h_*\tau^{\leq N}\cH((p \circ q)_* q^* M))) 
\cong \cH^n(h_*\tau^{\leq N}(p \circ q)_* q^* M)) 
\]
Then using \cite[Corollary 5.2.14]{Dim04} again, we obtain the desired isomorphism.
\end{proof}
\subsection{Monodromic mixed Hodge modules}\label{ssec:MMHM}
Here we recall some basic properties of 
monodromic mixed Hodge modules. We do not give the precise definition here and refer the reader to \cite[\S 2]{DM20} and \cite[\S 2.9]{BBBBJ15} for the detailed discussion. 
Let $X$ be a separated scheme locally of finite type over complex number whose connected components are quasi-compact. 
Then we can define an abelian category $\MMHM(X)$ 
of \textbf{monodromic mixed Hodge modules} on $X$.
Roughly speaking, a monodromic mixed Hodge module consists of its underlying mixed Hodge module $M$ and a monodromy operator acting on it.
We have a natural functor
\[
\MMHM(X) \to \MHM(X)
\]
forgetting the monodromy operator and a fully faithful functor
\[
\MHM(X) \hookrightarrow \MMHM(X)
\]
which associates a mixed Hodge module $M$ to a monodromic mixed Hodge module whose underlying mixed Hodge module is $M$ and the monodromy operator is trivial.
As similar to the usual mixed Hodge modules, 
monodromic mixed Hodge modules are also equipped with weight filtrations.

The bounded derived category $D^b(\MMHM(X) )$ 
admits a six-functor formalism, 
similarly to the usual mixed Hodge modules. 
The inclusion functor $D^b(\MHM(X) ) \to D^b(\MMHM(X) )$ is compatible with these six operations. 
The forgetful functor $D^b(\MMHM(X) ) \to D^b(\MHM(X) )$ is compatible with four operations $f_*, f_!, f^*, f^!$ for a morphism $f$ between separated finite type complex schemes.
However, the tensor product of monodromic mixed Hodge modules is not compatible with the tensor product of the underlying mixed Hodge modules.

For a regular function $f \colon X \to \bA^1$, 
we can define the \textbf{monodromic vanishing cycle functor} for (possibly unbounded) mixed Hodge module complexes
\begin{equation} \label{eq:defmonvan}
\varphi^{\mathrm{mmhm}}_f \colon D(\MHM(X)) \to D(\MMHM(X)),  
\end{equation}
which enhances the usual vanishing cycle functor by incorporating the monodromy operator acting on it.
 
The essential difference between monodromic 
and the usual mixed Hodge modules are the following: 

\begin{itemize}
    \item \textbf{Thom--Sebastiani isomorphism} holds 
    for the monodromic vanishing cycle functors 
    \eqref{eq:defmonvan}. 
    See \cite[Proposition 2.13]{DM20} for the precise statement 
    and other basic properties. 
    
    \item There exists an object 
    $\bL^{1/2} \in D^b(\MMHM(\pt))$ with an isomorphism 
    \[
    (\bL^{1/2})^{\otimes 2} 
    \cong \bL \in D^b(\MMHM(\pt)), 
    \]
    where we put $\bL\coloneqq \mH_c^*(\bA^1, \bQ) 
    \in D^b(\MHM(\pt)) \subset D^b(\MMHM(\pt))$, 
    which is concentrated in cohomological degree two, 
    and is pure of weight two. 
\end{itemize}

When $X$ is an irreducible variety, 
we define the object $\sIC_X \in \MMHM(X)$ as follows: 
\[
\sIC_X\coloneqq \bL^{-\dim X/2} \otimes \widetilde{\sIC}_X 
\in \MMHM(X), 
\]
where $\widetilde{\sIC}_X$ denotes the intermediate extension 
of $\bQ_{X_{\reg}}$ on the regular locus $X_{\reg} \subset X$. 
We will also use the following object: 
\begin{equation} \label{eq:defBCvir}
\mH^*(\mathrm{B}\bC^*)_{\vir}\coloneqq 
\bL^{1/2} \otimes \mH^*(\mathrm{B}\bC^*) 
\in D(\MMHM(\pt)). 
\end{equation}

As in the previous subsection, 
we can define the notion of monodromic mixed Hodge modules 
for an Artin stack. 
In particular, we can define the object 
$\sIC_\fX \in \MMHM(\fX)$ for a smooth Artin stack $\fX$. 
Moreover, we can define the functor 
\[
\mcH^n(p_*(-) ) \colon 
\MMHM(\fX) \to \MMHM(X) 
\]
for a morphism $p \colon \fX \to X$ 
from an Artin stack $\fX$ to a scheme $X$ 
satisfying the condition (*) in \S \ref{ssec:MHM}. 

Let $X$ be a separated scheme locally of finite type over $\bC$ whose connected components are quasi-compact.
We say that a (possibly unbounded) complex $M \in D(\MMHM(X))$ is \textbf{locally finite} if for each connected component $Z \subset X$, 
the following conditions hold: 
\begin{itemize}
    \item For each $n \in \bZ$, the set $\{i\in \bZ \mid \mathrm{gr}_n^{W} \cH^i(M) |_Z \neq 0 \}$ is finite.
    \item There exists an integer $n$ such that $W_n \cH^i(M) |_Z =0$ for all $i$.
\end{itemize}
We let $D^{\geq, \mathrm{lf}}(\MMHM(X)) \subset D^{\geq}(\MMHM(X))$ denote the full subcategory consisting of locally finite monodromic mixed Hodge complexes.
We can see that the Grothendieck group $K_0(D^{\geq, \mathrm{lf}}(\MMHM(X)))$ is isomorphic to the completion of $K_0(\MMHM(X))$ with respect to ideals $\{I_i \}_{i\in \bZ}$ where $I_i$ is generated by objects whose weight is greater than $i$.

Let $(X, m)$ be a monoid scheme where $X$ is a separated and locally of finite type over complex number whose connected components are quasi-compact and $m \colon X \times X \to X$ is a finite morphism.
For objects $M, N \in D^{\geq, \mathrm{lf}}(\MMHM(X))$,
we define
\[
M \boxtimes_m N \coloneqq m_*(M \boxtimes N) \in D^{\geq, \mathrm{lf}}(\MMHM(X)).
\]
The functor $\boxtimes_m$ defines a symmetric monoidal structure on the category $D^{\geq, \mathrm{lf}}(\MMHM(X))$.
Therefore for each $n \in \bZ_{>0}$, we can define the symmetric product functor
\[
\Sym_{\boxtimes_m}^{n} \colon D^{\geq, \mathrm{lf}}(\MMHM(X)) \to D^{\geq, \mathrm{lf}}(\MMHM(X)).
\]

\subsection{Approximation by proper morphisms}
\label{sec:appro}
Let $L$ be a line bundle on a smooth projective curve $C$ 
with $\deg(L) > 2g(C)-2$, 
and put $Y \coloneqq \Tot_C(L)$. 
For given integers $r, m \in \bZ$ with $r>0$, 
we denote by 
$p \colon \fM^{\semist}_Y(r, m) \to M_Y^{\semist}(r, m)$ 
the morphism from the moduli stack to its good moduli space. 
We fix a regular function 
$F \colon M^{\semist}_Y(r, m) \to \bA^1$ 
and denote by 
$\widetilde{F}\coloneqq F \circ p \colon 
\fM^{\semist}_Y(r, m) \to \bA^1$ 
the composition. 

Let us recall the construction of 
moduli spaces of {\it framed objects} 
following \cite{DM20, MS20, Mei15}. 
We follow the notations in \cite[\S 3.3]{MS20}. 
By construction, we have 
$\fM^{\semist}_Y(r, m)=\left[\Quot^{\semist}/\GL_n \right]$, 
where 
\[
\Quot^{\semist} \subset \Quot 
\]
is the GIT semistable locus 
inside a certain quot scheme $\Quot$ 
with respect to a certain $\GL_n$-linearization. 
For a given integer $f > n$, we put 
\[
\bA\coloneqq\Hom(\bC,\bC^n )^f, 
\quad G_n\coloneqq\bC^* \times \GL_n. 
\]
We have a $G_n$-action on $\Quot$ which passes through 
the $\GL_n$-action. 
We define a $G_n$-action on $\bA$ as follows: 
\[
(t, g) \cdot (a_i) \coloneqq (t^{-1}a_ig), \quad 
(t, g) \in G_n, (a_i)_{i=1}^f \in \bA. 
\]

By choosing certain $G_n$-linearizations on $\bA$ 
and $\Quot \times \bA$, 
we obtain the diagram 
\begin{equation} \label{eq:appro}
    \xymatrix{
    &U_f
    \ar@{^{(}->}[r] \ar[rd]_{\kappa} 
    &M_f \ar@{^{(}->}[r] \ar[d]_-{\pi_f} 
    &\mcX_f \ar[ld] \\
    &  &M^{\semist}_Y(r, m), &
    }
\end{equation}
where we put 
\begin{align*}
&U_f\coloneqq(\Quot^{\semist} \times \bA^{\semist} )/\PG_n, 
\quad M_f\coloneqq(\Quot \times \bA )^{\semist}/\PG_n, \\
&\mcX_f\coloneqq[\Quot^{\semist} \times \bA/\PG_n].
\end{align*}
The diagram \eqref{eq:appro} satisfies 
the following properties (cf. \cite[Proposition 3.6]{MS20}): 
\begin{itemize}
    \item the horizontal morphisms are open immersions, 
    \item $U_f$ and $M_f$ are smooth schemes and 
    the morphism $\pi_f$ is projective, 
    \item we have 
    $\lim_{f \to \infty}\codim_{\mcX_f}(\mcW_f)= \infty$, 
    where $\mcW_f \subset \mcX_f$ denotes 
    the complement of $U_f$.  
\end{itemize}

By the following proposition, 
we can compute the cohomology objects 
$\mcH^n(
p_*\varphi^{\mathrm{mmhm}}_{\widetilde{F}}\sIC_{\fM^{\semist}_Y(r, m)} 
)$ 
using the push-forward 
along the proper morphism $\pi_f$: 

\begin{prop}[cf. {\cite[Lemma 4.1, Proposition 4.3]{DM20}}] 
\label{prop:appro}
The following statements hold: 
\begin{enumerate}
    \item For each $n \in \bZ$, 
    there exists $f \gg 0$ such that 
    \[
    \mcH^n( p_*\varphi^{\mathrm{mmhm}}_{\widetilde{F}}
    \sIC_{\fM^{\semist}_Y(r, m)} 
    ) \cong 
    \mcH^n( \pi_{f*}\varphi^{\mathrm{mmhm}}_{F \circ \pi_f}\sIC_{M_f} 
    ).
    \]
    
    \item We have an isomorphism 
    \[
    \mcH(
p_*\varphi^{\mathrm{mmhm}}_{\widetilde{F}}
    \sIC_{\fM^{\semist}_Y(r, m)} 
    ) \cong 
    \varphi^{\mathrm{mmhm}}_{F}\mcH(
p_*\sIC_{\fM^{\semist}_Y(r, m)} 
    ).
    \]
\end{enumerate}
\end{prop}
\begin{proof}
We just give an outline of the proof. 
See \cite[Lemma 4.1, Proposition 4.3]{DM20} for the details. 
Using the fact 
$\lim_{f \to \infty}\codim_{\mcX_f}
(\mcW_f )=\infty$, 
we can check that 
the morphism $\pi_f \colon M_f \to M^{\semist}_Y(r, m)$ 
approximates the map $p \colon \fM^{\semist}_Y(r, m) \to M^{\semist}_Y(r, m)$ in the sense of (*) in \S \ref{ssec:MHM}. 
Hence the first assertion holds. 

The second assertion now follows from 
the natural isomorphism 
\[
\varphi^{\mathrm{mmhm}}_{F} \circ \pi_{f*} \cong 
\pi_{f*} \circ \varphi^{\mathrm{mmhm}}_{\pi_f \circ F}
\]
between functors, 
which holds since the morphism 
$\pi_f \colon M_f \to M^{\semist}_Y(r, m)$ 
is proper. 
\end{proof}

The following statement will be used in \S \ref{ssec:indHiggs}.

\begin{prop}\label{prop:decompthm}
Let $\fZ \subset \fM^{\semist}_Y(r, m)$ be the critical locus of $\widetilde{F}$ and $Z \subset M^{\semist}_Y(r, m)$ 
be its good moduli space.
Given a morphism $q \colon Z \to W$ between schemes, 
we have an isomorphism
\[
   \cH ( q_*\mcH(
     (p|_{\fZ})_*\varphi^{\mathrm{mmhm}}_{\widetilde{F}}\sIC_{\fM^{\semist}_Y(r, m)} 
    )) 
    \cong 
       \mcH(
    (q \circ p|_{\fZ})_*\varphi^{\mathrm{mmhm}}_{\widetilde{F}}\sIC_{\fM^{\semist}_Y(r, m)} 
    )
\]
\end{prop}

\begin{proof}
Fix an integer $n$.
We let $Z_f$ denote the critical locus of the function $F \circ \pi_f $. Take a sufficiently large integer $f$ such that the following isomorphism holds:
\begin{align*}
        \mcH^n(
     (q \circ p|_{\fZ})_*\varphi^{\mathrm{mmhm}}_{\widetilde{F}}\sIC_{\fM^{\semist}_Y(r, m)} 
    ) \cong 
    \mcH^n(
   (q \circ \pi_{f}|_{Z_f})_* \varphi^{\mathrm{mmhm}}_{F \circ \pi_f}\sIC_{M_f} 
    ).
\end{align*}
 We have the following isomorphisms
 \begin{align*}
    \mcH^n(
   (q \circ \pi_{f}|_{Z_f})_* \varphi^{\mathrm{mmhm}}_{F \circ \pi_f}\sIC_{M_f} 
    )
    &\cong 
    \mcH^n(
   q_* \varphi^{\mathrm{mmhm}}_{F \circ \pi_f}(\pi_{f*}\sIC_{M_f}) 
    ) \\
    &\cong \mcH^n(
   q_* \varphi^{\mathrm{mmhm}}_{F \circ \pi_f}(\cH(\pi_{f*} \sIC_{M_f})) 
    ) \\
    &\cong \mcH^n(
   q_* \cH(\varphi^{\mathrm{mmhm}}_{F \circ \pi_f}(\pi_{f*} \sIC_{M_f})) 
)
 \end{align*}
 where the second isomorphism follows from Saito's decomposition theorem.
 If $f$ is sufficiently large, we also have an isomorphism
 \[   \mcH^n(
   q_* \cH(\varphi^{\mathrm{mmhm}}_{F \circ \pi_f}(\pi_{f*} \sIC_{M_f})) 
    )  
    \cong \cH^n ( q_*\mcH(
     p_*\varphi^{\mathrm{mmhm}}_{\widetilde{F}}\sIC_{\fM^{\semist}_Y(r, m)} 
    )) 
 \]
 so we obtain the claim.
\end{proof}

\subsection{Cohomological integrality theorem for $L$-Higgs bundles}\label{ssec:intL}
Here we prove the cohomological integrality theorem 
for $Y \coloneqq \Tot_C(L)$, 
where $L$ is a line bundle on a smooth projective curve $C$ 
with $\deg(L) > 2g(C)-2$. 
Since the category of $\mu$-semistable one-dimensional sheaves 
on $Y$ is homological dimension one, 
this can be proved in the same manner as \cite[Theorem A]{DM20} by applying the main result of \cite{Mei15}. 
However we give a sketch of the proof for reader's convenience.

For a given rational number $\mu \in \bQ$, 
we set
\[
M^{\semist}_Y(\mu) \coloneqq 
\coprod_{\frac{m}{r}=\mu} M^{\semist}_Y(r, m). 
\]
For each positive integer $n \in \bZ_{>0}$, 
we have the following morphism: 
\[
\oplus \colon 
(M^{\semist}_Y(\mu) )^{\times n} 
\to M^{\semist}_Y(\mu), 
\quad (E_i)_{i=1}^n \mapsto \oplus_i E_i, 
\]
which is finite 
(cf. \cite[Examples 2.14 and 2.16]{DM15}). 
We define functors 
\[
\Sym^n_{\boxtimes_\oplus}, \quad 
\Sym_{\boxtimes_\oplus} \colon 
D^{\geq, lf}(\MMHM(M^{\semist}_Y(\mu)) ) \to 
D^{\geq, lf}(\MMHM(M^{\semist}_Y(\mu)) )
\]
as follows: 
\[
\Sym^n_{\boxtimes_\oplus}(\mcF)\coloneqq
(\oplus_* \mcF^{\boxtimes n} )^{\fS_n}, 
\quad 
\Sym_{\boxtimes_\oplus}(\mcF)\coloneqq
\oplus_{n \geq 1} \Sym^n_{\boxtimes_\oplus}(\mcF). 
\]

\begin{prop} \label{prop:sym}
The following statements hold: 
\begin{enumerate}
    \item The functor $\Sym_{\boxtimes_\oplus}$ is exact. 
    
    \item The functor $\Sym_{\boxtimes_\oplus}$ 
    sends pure objects to pure objects. 
    
    \item Let $F_\mu \colon M^{\semist}_Y(\mu) \to \bA^1$ be 
    a regular function satisfying 
    $F_\mu(A \oplus B)=F_\mu(A)+F_\mu(B)$ 
    for all $A, B \in M^{\semist}_Y(\mu)$. 
    Then the functors 
    $\varphi^{\mathrm{mmhm}}_{F_{\mu}}$ and 
    $\Sym_{\boxtimes_\oplus}$ commute. 
\end{enumerate}
\end{prop}
\begin{proof}
The same proofs as in 
\cite[Propositions 3.5, 3.8, 3.11]{DM20} work 
by using the finiteness of the morphism 
$\oplus \colon 
(M^{\semist}_Y(\mu) )^{\times n} 
\to M^{\semist}_Y(\mu)$ 
and Thom--Sebastiani isomorphism for 
the vanishing cycle functors $\varphi^{\mathrm{mmhm}}_{(-)}$. 
\end{proof}

We use the following notations: 
\[
\sIC_{\fM^{\semist}_Y(\mu)}
\coloneqq \bigoplus_{\frac{m}{r}=\mu} \sIC_{\fM^{\semist}_Y(r, m)}, \quad 
\sIC_{M^{\semist}_Y(\mu)}
\coloneqq \bigoplus_{\frac{m}{r}=\mu} \sIC_{M^{\semist}_Y(r, m)}. 
\]
Recall that we denote by 
$p \colon \fM^{\semist}_Y(r, m) \to M^{\ss}_Y(r, m)$ 
the canonical morphism to the good moduli space. 
Recall also that the definition of the object 
$\mH^*(\mathrm{B}\bC^*)_{\vir}$ from \eqref{eq:defBCvir}. 

\begin{thm}\label{thm:int-IC}
We have the following isomorphisms 
in $D^{\geq, lf}(\MMHM(M^{\semist}_Y(\mu) ))$: 
\begin{align} 
&\mcH(p_*\sIC_{\fM^{\semist}_Y(\mu)} ) 
\cong \Sym_{\boxtimes_\oplus}\left( 
\mH^*(\mathrm{B}\bC^*)_{\vir} \otimes 
\sIC_{M^{\semist}_Y(\mu)}
\right), \label{eq:intIC}\\
    &\mcH(
     p_* \varphi^{\mathrm{mmhm}}_{F_{\mu} \circ p}
    \sIC_{\fM^{\semist}_Y(\mu)} 
    ) 
    \cong \Sym_{\boxtimes_\oplus}\left( 
    \mH^*(\mathrm{B}\bC^*)_{\vir} \otimes 
    \varphi^{\mathrm{mmhm}}_{F_{\mu}}\sIC_{M^{\semist}_Y(\mu)}
    \right) \label{eq:intBPS}
\end{align} 
for a regular function 
$F_\mu \colon M^{\semist}_Y(\mu) \to \bA^1$ 
satisfying $F_\mu(A \oplus B)=F_\mu(A)+F_\mu(B)$ 
for all $A, B \in M^{\semist}_Y(\mu)$. 
\end{thm}
\begin{proof}
We first construct the isomorphism 
\eqref{eq:intIC}. 
By the exactness of the functor 
$\Sym_{\boxtimes_\oplus}$ 
(see Proposition \ref{prop:sym} (1)), 
the right hand side is isomorphic to its total cohomology. 
Hence it is enough to prove the isomorphism 
for each cohomology. 
By Proposition \ref{prop:appro} (1), 
for each $n \in \bZ$ and 
$(r, m) \in \bZ_{>0} \times \bZ$, 
there exists $f \gg 0$ such that 
we have an isomorphism 
\begin{equation} \label{eq:Hnp*}
\mcH^n(p_*\sIC_{\fM^{\semist}_Y(r, m)} ) 
\cong \mcH^n(\pi_{f*}\sIC_{M_f} ). 
\end{equation}
Since the morphism $\pi_f$ is proper and 
the object $\sIC_{M_f}$ is pure, 
it follows that the object in \eqref{eq:Hnp*} 
is a pure mixed Hodge module. 

On the other hand, 
since $\sIC_{M^{\semist}_Y(r, m)}$ is pure, 
Proposition \ref{prop:sym} (2) implies that 
the $n$-th cohomology of the right hand side of 
\eqref{eq:intIC} is also pure. 

Now $n$-th cohomology of both sides of \eqref{eq:intIC} 
are direct sums of simple pure mixed Hodge modules. 
Hence it is enough to prove the equality 
in the Grothendieck group 
$K_0(D^{\geq, lf}(M^{\semist}_Y(\mu) ) )$, 
which holds by the main theorem of \cite{Mei15}. 

The second isomorphism \eqref{eq:intBPS} 
follows by applying the vanishing cycle functor 
$\varphi^{\mathrm{mmhm}}_{F_\mu}$ to both sides of 
the isomorphism \eqref{eq:intIC} 
and then using 
Proposition \ref{prop:appro} (2) 
and Proposition \ref{prop:sym} (3). 
\end{proof}


We end this section 
by proving Proposition \ref{prop:BPSvsVan} 
in the previous section: 
\begin{proof}[Proof of Proposition \ref{prop:BPSvsVan}]
Fix integers $r>0$ and $m \in \bZ$, 
and put $\mu\coloneqq m/r$. 
Let $g \colon B_Y \to \bA^1$ be the function 
defined in (\ref{eq:fctBS}). 
Recall from Proposition \ref{prop:dchart} 
that we have 
\[
\fM^{\semist}_X(r, m) \cong 
\{d(\widetilde{h}_Y \circ g )=0 \} 
\subset \fM^{\semist}_Y(r, m) 
\]
for a line bundle $L_2$ 
with $\deg(L_2) \gg 0$. 
Hence the first isomorphism in Proposition \ref{prop:BPSvsVan} 
follows from Proposition \ref{prop:joyce=van}. 

For the second isomorphism, let us put 
\[
F_\mu\coloneqq g \circ h_Y \colon M^{\semist}_Y(\mu) \to \bA^1. 
\]
By taking the first cohomology of 
the isomorphism \eqref{eq:intBPS}, 
we obtain 
\[
\mcH^1(
 p_* \varphi^{}_{F_\mu \circ p}\sIC_{\fM^{\semist}_Y(\mu)}
)
\cong \varphi^{}_{F_\mu}\sIC_{M^{\semist}_Y(\mu)}. 
\]
Restricting it to the component 
$M^{\semist}_Y(r, m) \subset M^{\semist}_Y(\mu)$, 
we get the second isomorphism 
in Proposition \ref{prop:BPSvsVan}. 
\end{proof}

\begin{rmk}
It is clear from the proof that Proposition \ref{prop:BPSvsVan} naturally extends to an isomorphism of monodromic mixed Hodge modules.
\end{rmk}

\section{The case of Higgs bundles}
\label{sec:Higgs}
In this section, we prove the cohomological integrality theorem and the cohomological $\chi$-independence theorem 
for Higgs bundle moduli spaces on curves 
using the dimensional reduction theorem 
due to the first author \cite{Kin21}. 

\subsection{Dimensional reduction theorem}

Let $\bs{\fY}$ be a quasi-smooth derived Artin stack and $\bfT^*[-1]\bs{\fY}$ be its $(-1)$-shifted cotangent stack.
We write $\fY \coloneqq t_0(\bs{\fY})$ and $\widetilde{\fY} \coloneqq t_0(\bfT^*[-1]\bs{\fY})$, and $\pi \colon \widetilde{\fY} \to \fY$ be the natural projection.
As we have seen in Example \ref{ex:shiftedcot}, $\bfT^*[-1]\bs{\fY}$ carries a natural $(-1)$-shifted symplectic structure. Further, as is
proved in \cite[Lemma 3.3.3]{Tod19}, there exists a natural orientation 
\begin{equation}\label{eq:canori}
o \colon \det{\bL_{\bfT^*[-1]\bs{\fY}}} |_{\fY^{\red}} \cong ({\pi^{\red}})^* \det(\bL_{\bs{\fY}})^{\otimes {2}}.
\end{equation}
We let $\varphi_{\bfT^*[-1]\bs{\fY}}$ denote the perverse sheaf on $\widetilde{\fY}$ recalled in \S \ref{ssec:van} with respect to this $(-1)$-shifted symplectic structure and orientation.
The following theorem is called the \textbf{dimensional reduction theorem}.

\begin{thm}[{\cite[Theorem 4.14]{Kin21}}]\label{thm:dimred}
There exists a natural isomorphism in $D^b_c(\fY)$
\begin{equation}\label{eq:dimred}
\pi_* \varphi_{\bfT^*[-1]\bs{\fY}} \cong \bD\bQ_{\fY}[-\vdim \bs{\fY}].
\end{equation}
Here $\vdim \bs{\fY} \coloneqq \rank \mathbb{L}_{\bs{\fY}}$ denotes the virtual dimension of $\bs{\fY}$.
\end{thm}

We now discuss the generalization of this theorem to the level of complexes in  mixed Hodge modules.
Firstly we discuss the case when $\bs{\fY}$ is a derived scheme. To specify that $\bs{\fY}$ is schematic, we write $\bs{Y} = \bs{\fY}$, $Y = \fY$ and $\widetilde{Y} = \widetilde{\fY}$.
The following lemma is useful:

\begin{lem}\label{lem:isomhodge}
Let $X$ be an algebraic variety and take complexes of mixed Hodge module $M, N \in D^b(\mathrm{MHM}(X))$ such that there exists an isomorphism
$\eta \colon \rat(M) \cong \rat(N)$ in $D^b(X)$.
Assume that for each $i <0$ the group $\Ext^i(\rat(M), \rat(N))$ vanishes and we have an isomorphism of mixed Hodge structures $\mH^0(X, \sHom(M, N)) \cong \bQ$.
Then $\eta$ naturally extends to an isomorphism $M \cong N$ in $D^b(\mathrm{MHM}(X))$.
\end{lem}

\begin{proof}
Consider the natural map of mixed Hodge complexes
\[
\tau_{\leq 0}\RHom(M, N) \to \RHom(M, N).
\]
The assumption implies an isomorphism $\bQ \cong \tau_{\leq 0}\RHom(M, N)$ hence we obtain a map in $D^b(\mathrm{MMHM}(X))$
\[
\bQ_X \to  \sHom(M, N)
\]
by adjunction.
Then the following composition of morphisms in $D^b(\mathrm{MHM}(X))$ 
\[
M = M \otimes \bQ_X \to M \otimes \sHom(M, N) \to N
\]
upgrades the isomorphism $\eta$ up to scalar.
\end{proof}

\begin{prop}\label{prop:dimredsch}
Assume that the virtual dimension $\vdim \bs{Y}$ is even.
Then the dimensional reduction isomorphism \eqref{eq:dimred} for a quasi-smooth derived scheme $\bs{Y}$ naturally upgrades to an isomorphism in $D^b(\MHM(Y))$: 
\begin{equation}\label{eq:dimredMHM}
\pi_* \varphi_{\bfT^*[-1]\bs{Y}}^{\mathrm{mhm}} \cong \bL^{\vdim \bs{Y}/2} \otimes \bD\bQ_{Y}.
\end{equation}
\end{prop}

\begin{proof}
Using Lemma \ref{lem:isomhodge}, we only need to prove that the mixed Hodge structure on
\[
\Hom(\pi_* \varphi_{\bfT^*[-1]\bs{Y}}, \bD\bQ_Y[-\vdim \bs{Y}])
\]
is pure of weight zero.
As this statement can be checked locally, using \cite[Theorem 4.1]{BBJ19}, 
we may assume that there exists a smooth scheme $U$ which admits a global \'etale coordinate, a trivial vector bundle $E$ on $U$, and a section $s \in \Gamma(U, E)$ such that  $\bs{Y}$ is isomorphic to the derived zero locus $\bs{Z}(s)$.
In this case the proof of \cite[Theorem 3.1]{Kin21} shows that the dimensional reduction isomorphism \eqref{eq:dimred} can be identified with Davison's local dimensional reduction theorem \cite[Theorem A.1]{Dav17}. As the proof of this theorem works verbatim for complexes of mixed Hodge modules, we conclude that the claim holds. 
\end{proof}

\begin{rmk}\label{rmk:dimredmmhm1}
We expect that the isomorphism \eqref{eq:dimredMHM} further upgrades to an isomorphism $D^b(\MMHM(Y))$
\[
\pi_* \varphi_{\bfT^*[-1]\bs{Y}}^{\mathrm{mmhm}} \cong \bL^{\vdim \bs{Y}/2} \otimes \bD\bQ_{Y}.
\]
 However, we could not prove this since we do not know whether the tensor-hom adjunction holds for monodromic mixed Hodge modules.
Instead, we can easily see that
we have an isomorphism in $D^b(\MMHM(Y))$
\[
\cH(\pi_* \varphi_{\bfT^*[-1]\bs{Y}}^{\mathrm{mmhm}}) \cong \cH(\bL^{\vdim \bs{Y}/2} \otimes \bD\bQ_{Y})
\]
since the monodromy operator acts trivially on both sides (see \cite[Remark 3.9]{Dav17b}).
It is enough for our purposes.

\end{rmk}

Now we discuss the stacky case of this proposition.
Let $\bs{\fY}$ be a quasi-smooth derived Artin stack such that its classical truncation $\fY = t_0(\bs{\fY})$ is of the form $[Y / G]$ for some scheme $Y$ and a linear algebraic group $G$. In this case, we can upgrade the dimensional reduction theorem to an isomorphism of mixed Hodge structures.

\begin{prop}
Assume that $\vdim \bs{\fY}$ is even. Then the dimensional reduction isomorphism 
$\mH^*(\widetilde{\fY}, \varphi_{\bfT^*[-1]
\bs{\fY}} ) \cong \HBM_{-* + \vdim \bs{\fY}}(\fY)$
upgrades naturally to an isomorphism of mixed Hodge structures
\[
\mH^*(\widetilde{\fY}, \varphi_{\bfT^*[-1]
\bs{\fY}}^{\mathrm{mhm}} ) \cong \bL^{\vdim \bs{\fY}/2} \otimes \HBM_{-*}(\fY)
\]
\end{prop}

\begin{proof}
For a fixed $i$, take a smooth morphism $q \colon T \to  \fY$ of relative dimension $d$ such that  
\[
q^! \colon \HBM_{-i + \vdim \bs{\fY}}(\fY) \to \HBM_{-i + \vdim \bs{\fY} + 2d}(T)
\]
and the map
\[
\tilde{q}^! \colon \mH^i(\widetilde{\fY}, \varphi_{\bfT^*[-1] \bs{\fY}} ) \to \mH^i(\widetilde{T}, \tilde{q}^* \varphi_{ \bfT^*[-1] \bs{\fY}} )
\]
are isomorphisms. Here $\tilde{q} \colon \widetilde{T} \to  \widetilde{\fY}$ is the base change of $q$.
Therefore we need to show that the following composition of isomorphism of vector spaces 
\begin{align*}
    \mH^i(\widetilde{T}, \tilde{q}^* \varphi_{ \bfT^*[-1] \bs{\fY}} ))
    &\cong \mH^i(\widetilde{\fY}, \varphi_{\bfT^*[-1] \bs{\fY}})  \\
    &\cong \HBM_{-i + \vdim \bs{\fY}}(\fY) \\
    &\cong \HBM_{-i + \vdim \bs{\fY} + 2d}(T)
\end{align*}
upgrades to an isomorphism of mixed Hodge structures.
To prove this, we will show that the following morphism in $D^b(T)$
\begin{align*}
\eta_{q} \colon {\pi_T}_* \tilde{q}^* \varphi_{ \bfT^*[-1] \bs{\fY}} 
& \cong q^* {\pi_{\fY}}_*  \varphi_{ \bfT^*[-1] \bs{\fY}} \\
& \cong q^* \bD \bQ_{\fY} [- \vdim \bs{\fY}] \\
&\cong \bD \bQ_{T} [- \vdim \bs{\fY} - 2d ] 
\end{align*}
upgrades an isomorphism in $D^b(\MHM(T))$. Here $\pi_{T} \colon \widetilde{T} \to T$
and $\pi_{\fY} \colon \widetilde{\fY} \to \fY$ are natural projections and the second isomorphism is the dimensional reduction isomorphism.
Using Lemma \ref{lem:isomhodge}, we need to show that the mixed Hodge structure of 
\begin{equation}\label{eq:wtzero}
\Hom({\pi_T}_* \tilde{q}^* \varphi_{ \bfT^*[-1] \bs{\fY}}, \bD \bQ_{T} [- \vdim \bs{\fY} - 2d ])
\end{equation}
is weight zero.
To prove this, take a smooth surjective morphism from a derived scheme
$\bs{h} \colon \bs{U} \to \bs{\fY}$ such that $\vdim \bs{U}$ is even. Write $h = t_0(\bs{h})$, $t_0(\bs{U}) = U$, and $\widetilde{X} \coloneqq t_0(\bfT^*[-1]\bs{U})$.
Let $\widetilde{T \times_{\fY} U}$ be the fibre product $(T \times_{\fY} U) \times_{U} \widetilde{U}$. We let $\pi_{T \times_{\fY} U} \colon \widetilde{T \times_{\fY} U} \to T \times_{\fY} U$ and $\pi_{U} \colon \widetilde{U} \to U$ denote the natural projections and $q_{U} \colon T \times_{\fY} U \to U$ and $\tilde{q}_U \colon \widetilde{T \times_{\fY} U} \to \widetilde{U}$ be the base changes of $q$. 
Then we can construct a natural isomorphism
\[
\eta_{q_U} \colon {\pi_{T \times_{\fY} U}}_* \tilde{q}_U^* \varphi_{ \bfT^*[-1] \bs{U}} \cong \bD \bQ_{T \times _{\fY} U}[- \vdim \bs{U} -2d]
\]
in the same manner as $\eta_q$.
As we have seen in Proposition \ref{prop:dimredsch}, the map $\eta_{q_U}$ upgrades to an isomorphism in $D^b(\MHM(T \times_{\fY} U))$
\[
{\pi_{T \times_{\fY} U}}_* \tilde{q}_U^* \varphi_{ \bfT^*[-1] \bs{U}}^{\mathrm{mhm}} \cong \bL^{d + \vdim \bs{U}/2} \otimes \bD \bQ_{T \times _{\fY} U}.
\] 
Let $h_{\widetilde{\fY}} \colon \widetilde{U} \to \widetilde{\fY}$,  $h_T \colon T \times_{\fY} U \to T$ and $h_{T \times_{\fY} U} \colon \widetilde{T \times_{\fY} U} \to \widetilde{T}$ be the base changes of $h$.
Then we have a natural isomorphism
\[
h_{T}^* {\pi_T}_* \tilde{q}^* \varphi_{ \bfT^*[-1] \bs{\fY}}^{\mathrm{mhm}} \cong {\pi_{T \times_{\fY} U}}_* \tilde{q}_U^* h_{\widetilde{\fY}}^* \varphi_{ \bfT^*[-1] \bs{\fY}}^{\mathrm{mhm}} \cong \bL^{\dim h/2} \otimes  {\pi_{T \times_{\fY} U}}_* \tilde{q}_U^* \varphi_{ \bfT^*[-1] \bs{U}}^{\mathrm{mhm}}
\]
where the latter isomorphism follows from \cite[Proposition 4.10]{Kin21}. We also have a natural isomorphism
\[
 \bL^{d + \vdim \bs{\fY}/2} \otimes h_{T}^* \bD \bQ_{T} \cong \bL^{d + h/2 + \vdim \bs{U}/2} \otimes \bD \bQ_{T \times_{\fY} U}.
\]
Under these identifications, 
the proof of \cite[Theorem 4.14]{Kin21} implies that $h_T^* \eta_q$ is equal to $\eta_{q_U}[- \dim h]$ up to a certain choice of the sign. This and the fact that $\eta_{q_U}$ upgrades to an isomorphism in $D^b(\MHM(T \times_{\fY} U))$ imply that the weight of the mixed Hodge structure \eqref{eq:wtzero} is zero.
\end{proof}

The following statement can be proved in the same manner as the previous proposition:

\begin{prop}\label{prop:dimredMHM2}
We keep the notation from the previous proposition. Let $p \colon \fY \to Z$ be a morphism to a separated finite type complex scheme.
Then we have an isomorphism of mixed Hodge modules
\begin{equation}\label{eq:dimredMHM2}
\mcH((p \circ \pi_{\fY})_* \varphi_{\bfT^*[-1] \bs{\fY}}^{\mhm}) 
\cong  \bL^{\vdim \bs{\fY}/2} \otimes  \mcH(p_* \bD\bQ_{\fY}) 
\end{equation}
where $\pi_{\fY} \colon \widetilde{\fY} \to \fY$ is the natural projection.
\end{prop}

\begin{rmk}\label{rmk:monodromyfree}
The argument as in Remark \ref{rmk:dimredmmhm1} implies that the isomorphism \eqref{eq:dimredMHM2} upgrades to an isomorphism in $D^b(\MMHM(Z))$
\[
\mcH((p \circ \pi_{\fY})_* \varphi_{\bfT^*[-1] \bs{\fY}}^{\mmhm}) 
\cong  \bL^{\vdim \bs{\fY}/2} \otimes  \mcH(p_* \bD\bQ_{\fY}).
\]
\end{rmk}

\subsection{BPS cohomology for Higgs bundles}

In \cite{Dav16}, Davison defined BPS sheaves and BPS cohomology for preprojective algebras. In this section, we introduce Higgs counterpart of these notions. 

Let $C$ be a smooth projective curve of genus $g$. 
We write $S = \Tot_C(\omega_C)$ and 
$X = \Tot_C(\cO_C \oplus \omega_C)$.
Recall that $\fM_{X}^{\semist}(r, m)$ (resp. $\fM_{S}^{\semist}(r, m)$) denotes the moduli stack of one-dimensional semistable sheaves of rank $r$ and Euler characteristic $m$ on $X$ (resp. $S$), and 
$M_{X}^{\semist}(r, m)$ (resp. $M_{S}^{\semist}(r, m)$) denotes the good moduli space of $\fM_{X}^{\semist}(r, m)$ (resp. $\fM_{S}^{\semist}(r, m)$).
We have the following commutative diagram:
\[
\xymatrix@C=50pt{
\fM_{X}^{\semist}(r, m) \ar[d]^-{p_X} \ar[r]^-{\pi}
& \fM_{S}^{\semist}(r, m) \ar[d]^-{p_S} \\
M_{X}^{\semist}(r, m) \ar[r]^-{\bar{\pi}}
& M_{S}^{\semist}(r, m).
}
\]
It is shown in \cite[Theorem 5.1]{Kin21} that there exists a natural equivalence of $(-1)$-shifted symplectic derived Artin stacks $\bs{\fM}_{X}^{\semist}(r, m)  \cong \bfT^*[-1]\bs{\fM}_{S}^{\semist}(r, m)$, where $\bs{\fM}_{X}^{\semist}(r, m)$ (resp. $\bs{\fM}_{S}^{\semist}(r, m)$) denotes the derived enhancement of $\fM_{X}^{\semist}(r, m)$ (resp. $\fM_{S}^{\semist}(r, m)$).
Therefore \eqref{eq:canori} implies that there exists a canonical orientation
\[
o \colon (\pi^{\red, *} \det(\bL_{\bs{\fM}_{S}^{\semist}(r, m)} |_{\fM_{X}^{\semist}(r, m)^{\red}}))^{\otimes 2} \cong K_{\bs{\fM}_{X}^{\semist}(r, m)}^{\vir}.
\]
On the other hand, we have seen in Proposition \ref{prop:dchart} that there exist a line bundle $L$ with $\deg L >2g -2$ and a function $f$ on the moduli stack $\fM_{\Tot_C(L)}^{\ss}(r, m)$ of semistable sheaves on $\Tot_C(L)$ such that there exists an equivalence of $(-1)$-shifted symplectic derived Artin stacks $\bs{\fM}_{X}^{\semist}(r, m) \cong \DCrit(f)$.
Therefore there exists an orientation
\[
o' \colon K_{\fM_{\Tot(L)}(r, m)} |_{\fM_{X}^{\semist}(r, m)^{\red}} ^{\otimes {2}} \cong K_{\bs{\fM}_{X}^{\semist}(r, m)}^{\vir}.
\]

\begin{prop}\label{prop:compori}
There exists an isomorphism of orientations $o \cong o'$.
\end{prop}

\begin{proof}
We have seen in Proposition \ref{prop:CYori} that there exists a trivialization
\[
K_{\fM_{\Tot_C(L)}^{\ss}(r, m)} \cong \cO_{\fM_{\Tot_C(L)}^{\ss}(r, m)}.
\]
On the other hand, we also have a trivialization
\[
\det(\bL_{\bs{\fM}_{S}^{\semist}(r, m)} |_{\fM_{S}^{\semist}(r, m)^{\red}})) \cong \cO_{\fM_{S}^{\semist}(r, m)^{\red}},
\]
since there exists an open immersion
\[
\bs{\fM}_{S}^{\semist}(r, m) \hookrightarrow \bfT^* \fM_{C},
\]
where $\fM_{C}$ denotes the moduli stack of coherent sheaves on $C$.
Therefore we need to show that the following composition
\begin{align*}
\cO_{\fM_{X}^{\semist}(r, m)^{\red}}^{\otimes {2}}
&\cong (\pi^{\red, *} \det(\bL_{\bs{\fM}_{S}^{\semist}(r, m)} |_{\fM_{X}^{\semist}(r, m)^{\red}}))^{\otimes 2} \\
&\cong K_{\bs{\fM}_{X}^{\semist}(r, m)}^{\vir} \\
&\cong  K_{\fM_{\Tot_C(L)}^{\ss}(r, m)} |_{\fM_{X}^{\semist}(r, m)^{\red}} ^{\otimes {2}} \\
&\cong \cO_{\fM_{X}^{\semist}(r, m)^{\red}}^{\otimes {2}}
\end{align*}
has a square root.
More strongly, we will show that any invertible function 
\[
f \in \Gamma(\fM_{X}^{\semist}(r, m)^{\red}, \cO_{\fM_{X}^{\semist}(r, m)^{\red}}^{\times}) \cong \Gamma(M_{X}^{\semist}(r, m)^{\red}, \cO_{M_{X}^{\semist}(r, m)^{\red}}^{\times})
\]
is constant hence admits a square root.

We say that a reduced finite type complex scheme $X$ satisfies the property (P) if every invertible regular function on $X$ is locally constant.
What we need to prove is that the scheme $M_{X}^{\semist}(r, m)^{\red}$ satisfies the property (P).
Property (P) satisfies the following:
\begin{itemize}
    \item If we are given a surjective morphism between reduced finite type complex schemes  $X \to Y$ and $X$ satisfies the property (P), then $Y$ satisfies the property (P).
    \item For reduced finite type complex schemes $X$ and $Y$ satisfying the property (P), $X \times Y$ also satisfies the property (P).  
\end{itemize}

Write $k = \gcd(r, m)$ and $(r, m) = (k r_0, k m_0)$.
Take a partition $\lambda_{\bullet} = (\lambda_1, \lambda_2, \ldots, \lambda_t)$ of $k$, i.e., $\lambda_i$ is a positive integer with $\lambda_1 \geq \lambda_2 \geq \cdots \geq \lambda_t$ such that $\sum_i \lambda_i =k$ holds.
Let $M_{X}^{\semist}(r, m)^{\red}_{\lambda_{\bullet}}$ 
be the subscheme consisting of points corresponding to polystable sheaves which can be written as
\[
\bigoplus_i E_i, 
\]
where $E_i$ is a stable sheaf on $X$ such that  $\rank ({\pi_{X}}_* E_i) = \lambda_i r_0$, where $\pi_X \colon X \to C$ is the projection. We let $\overline{M_{X}^{\semist}(r, m)^{\red}_{\lambda_{\bullet}}}$ denote the closure of $M_{X}^{\semist}(r, m)^{\red}_{\lambda_{\bullet}}$.
Since we have an equality
\[
M_{X}^{\semist}(r, m)^{\red} = \bigcup_{\lambda_{\bullet}}
\overline{M_{X}^{\semist}(r, m)^{\red}_{\lambda_{\bullet}}},
\]
we need to show that the scheme $\overline{M_{X}^{\semist}(r, m)^{\red}_{\lambda_{\bullet}}}$ satisfies the property (P).

Consider the map
\[
\prod_{i=1}^t \overline{M_{X}^{\semist}(\lambda_i r_0, \lambda_i m_0)^{\red}_{(\lambda_i)}} \to M_{X}^{\semist}(r, m)^{\red}
\]
taking the direct sum.
The image of this map is nothing but $\overline{M_{X}^{\semist}(r, m)^{\red}_{\lambda_{\bullet}}}$.
Therefore it is enough to show that 
$\overline{M_{X}^{\semist}(\lambda_i r_0, \lambda_i m_0)^{\red}_{(\lambda_i)}}$ satisfies the property (P).
As we have an isomorphism
\[
\overline{M_{X}^{\semist}(\lambda_i r_0, \lambda_i m_0)^{\red}_{(\lambda_i)}} \cong 
M_{S}^{\semist}(\lambda_i r_0, \lambda_i m_0) \times \bA^1,
\]
we need to prove that $M_{S}^{\semist}(\lambda_i r_0, \lambda_i m_0)$ satisfies property (P).

Let $g \colon M_{S}^{\semist}(\lambda_i r_0, \lambda_i m_0) \to \bA^1$ be an invertible function.
We need to prove that $g$ is constant.
Let $h_S \colon M_{S}^{\semist}(\lambda_i r_0, \lambda_i m_0) \to B_S$ be the Hitchin fibration.  Since the general fiber of $h_S$ is connected, we have an isomorphism $h_{S, *} \cO_{M_{S}^{\semist}(\lambda_i r_0, \lambda_i m_0)} \cong \cO_{B_S}$. Therefore there exists an invertible function $g'$ on $B_S$ such that $g = g' \circ h_S$ holds.
Since $B_S$ is an affine space, $g'$ is a constant function.
\end{proof}

\begin{rmk}\label{rmk:compori}
The proof shows that any orientation
$o'' \colon L^{\otimes 2} \cong  K_{\bs{\fM}_{X}^{\semist}(r, m)}^{\vir}$
such that $L$ is trivial is isomorphic to $o$.
\end{rmk}

From now  we always equip
$\bs{\fM}_{X}^{\semist}(r, m)$ with the orientation $o$. 
Define a monodromic mixed Hodge module $\varphi_{M_{X}^{\semist}(r, m)}^{\mathrm{mmhm}}$ on $M_X^{\semist}(r, m)$ by
\[
\varphi_{M_{X}^{\semist}(r, m)}^{\mathrm{mmhm}} \coloneqq \cH^0(\bL^{-1/2} \otimes {p_X}_* \varphi_{\fM_{X}^{\semist}(r, m)}^{\mathrm{mmhm}}).
\]
For a given rational number $\mu$, we define
\begin{align*}
M^{\semist}_X(\mu) \coloneqq \coprod_{\frac{m}{r}=\mu} M^{\semist}_X(r, m), \ 
\varphi_{\fM_{X}^{\semist}(\mu)}^{\mathrm{mmhm}} \coloneqq \bigoplus_{\frac{m}{r}=\mu} \varphi_{\fM_{X}^{\semist}(r, m)}^{\mathrm{mmhm}} ,  \
 \varphi_{M_{X}^{\semist}(\mu)}^{\mathrm{mmhm}} &\coloneqq \bigoplus_{\frac{m}{r}=\mu} \varphi_{M_{X}^{\semist}(r, m)} ^{\mathrm{mmhm}}.
\end{align*}
Recall that we have constructed a symmetric monoidal structure
 $\boxtimes_{\oplus}$ on $D^{\geq, lf}(\MMHM(M^{\ss}_X(\mu)) )$
in \S \ref{ssec:MMHM}.
The following proposition is the cohomological integrality theorem (in the sense of \cite[Theorem A]{DM20}) for the Calabi--Yau threefold $X$:

\begin{prop}\label{prop:intHiggs}
We have an isomorphism
\[
\cH({p_X}_* \varphi_{\fM_{X}^{\semist}(\mu)}^\mathrm{mmhm}) \cong \Sym_{\boxtimes_\oplus}\left( 
\mH^*(\mathrm{B}\bC^*)_{\vir} \otimes \varphi_{M_{X}^{\semist}(\mu)}^{\mathrm{mmhm}}
\right)
\]
in $D^{\geq, lf}(\MMHM(M^{\ss}_X(\mu)) )$.
\end{prop}

\begin{proof}
Using Proposition \ref{prop:compori}, we may use the orientation $o'$ instead of $o$.
Then the claim follows from Proposition \ref{prop:isomvan} and Theorem \ref{thm:int-IC}.
\end{proof}

Now we state the Higgs version of the support lemma 
\cite[Lemma 4.1]{Dav16}:

\begin{prop}\label{prop:supplem}
Let $\ell \colon M_{S}^{\semist}(r, m) \times \bA^1 \to M_{X}^{\semist}(r, m)$ be the map given by
\[
M_{S}^{\semist}(r, m) \times \bA^1 \ni ([E], t) \mapsto [{i_t}_* E] \in M_{X}^{\semist}(r, m)
\]
where $i_t$ is the composition $S = S \times\{ t \} \hookrightarrow X$. Then the support of the perverse sheaf 
$\varphi_{M_{X}^{\semist}(r, m)}$ is contained in the image of $\ell$.
\end{prop}

The proof will be given in Appendix \ref{app:supplem}.

\begin{prop}
The monodromic mixed Hodge module $\varphi_{M_{X}^{\semist}(r, m)}^{\mmhm}$ is $\bA^1$-equivariant with respect to the natural $\bA^1$-action on $M_{X}^{\semist}(r, m)$.
Further there exists a  monodromic mixed Hodge module $\sBPS_{r, m} \in \MMHM(M_{S}^{\semist}(r, m))$ such that 
\[
\varphi_{M_{X}^{\semist}(r, m)}^{\mmhm} \cong \bL^{-1/2} \otimes l_* \pr_1^* \sBPS_{r, m}
\]
holds where $\pr_1 \colon M_{S}^{\semist}(r, m) \times \bA^1 \to M_{S}^{\semist}(r, m)$ is the projection.
\end{prop}

\begin{proof}
In general, a monodromic mixed Hodge module $M$ on $T \times \bA^1$ for an algebraic variety $T$ is $\bA^1$-equivariant if and only if the counit map $\pr_1^* {\pr_1}_* M \to M$ is isomorphic, where $\pr_1 \colon T \times \bA^1 \to T$ is the first order projection. Therefore the $\bA^1$-equivariance of $M$ is equivalent to the $\bA^1$-equivariance of $\rat(M)$. This is further equivalent to the condition $\sigma_1^*\rat(M) \cong \rat(M)$, where $\sigma_1 \colon T \times \bA^1 \cong T \times \bA^1$ is the map translating in the $\bA^1$-direction by $1 \in \bA^1$.

Now we return to the proposition. Let $\sigma_1 \colon \fM_{X}^{\semist}(r, m) \cong \fM_{X}^{\semist}(r, m)$ be the map induced by the translation map on $X = S\times \bA^1$ in the $\bA^1$-direction by $1 \in \bA^1$.
We need to show that there exists an isomorphism of perverse sheaves
\[
\sigma_1^* \varphi_{\fM_{X}^{\semist}(r, m)} \cong \varphi_{\fM_{X}^{\semist}(r, m)}.
\]
To do this, it is enough to show that there exists an isomorphism of orientations 
$\sigma_1^* o \cong o$
where $o$ is the natural orientation on $\fM_{X}^{\semist}(r, m)$.  But this is a consequence of Remark \ref{rmk:compori}.
The latter statement follows from Proposition \ref{prop:supplem}.
\end{proof}

The object $\sBPS_{r, m} \in \MMHM(M_{S}^{\semist}(r, m))$ is called the \textbf{BPS sheaf}. We will see that it is a pure Hodge module in the next section.
We write
\[
\BPS_{r, m} \coloneqq \mH^*(M_{S}^{\semist}(r, m), \sBPS_{r, m})
\]
and it is called the \textbf{BPS cohomology}.

\subsection{Cohomological integrality and $\chi$-independence for Higgs bundles}\label{ssec:indHiggs}

In this section, we prove the  $\chi$-independence theorem and cohomological integrality theorem for Higgs bundles using the dimensional reduction theorem.

We first need the following lemma:

\begin{lem}
The map $b |_{\mathrm{im}(h_X)} \colon \mathrm{im}(h_X) \to B_{Y}$ considered in Lemma \ref{lem:Hitfin} is injective for $X = \Tot_C(\cO_C \oplus \omega_C)$.  
\end{lem}

\begin{proof}
Let $\gamma, \gamma' \in \im(h_X)$ be cycles on $X$ such that the pushforward cycles $\sigma_* \gamma$ and $\sigma_* \gamma'$ define the same cycle on $Y = \Tot_C(L)$ where $\sigma$ is the projection from $X$ to $Y$.
We want to show $\gamma = \gamma'$.
Write 
\[
\gamma = \sum_i \gamma_{t_i}, \ \gamma' = \sum_i \gamma_{t_i'}'
\]
where $\gamma_{t_i}$ is supported on $S \times\{t_i\} \subset X$ and similarly for $\gamma'_{t_i'}$.
Take a point $p \in \Supp(\mathrm{coker}(\omega_C \hookrightarrow L))$.
Then the restriction of $\sigma$ at the fiber of $p$ is given by
\[
\bA^1 \oplus \omega_C|_p \to \bA^1 \cong L_p
\]
where the first map is the projection to the first factor and the latter map is induced from the composition $\cO_C \hookrightarrow \cO_C \oplus \omega_C \to L$.
Therefore the cycle $\sigma_* \gamma_{t_i} |_{L_p}$ is concentrated in $\{t_i\} \subset \bA^1 \cong L_p$.
Therefore we may assume that $\gamma$ and $\gamma'$ are contained in $S \times \{t \}$ for some $t \in \bA^1$.
Then the claim follows since the map $S \times \{t \} \subset X \to Y$ defines an injection on the set of cycles.
\end{proof}

The following corollary is an immediate consequence of the isomorphism \eqref{eq:isoBS} and the above lemma.

\begin{cor}\label{cor:chiHiggs}
 Let us take integers $r, m, m'$ such that $r>0$.
 Then there exists an isomorphism in $D^b(\MMHM(B_X))$: 
 \[
 h_{X*} \varphi_{M_{X}^{\semist}(r, m)}^{\mmhm} 
 \cong  h_{X*} \varphi_{M_{X}^{\semist}(r, m')}^{\mmhm}.
 \]
\end{cor}

\begin{cor} \label{cor:chiBPS}
Let $r, m, m'$ be as in the previous corollary. Then there exists an isomorphism in $D^b(\MMHM(B_S))$:
\[
 h_{S*} \sBPS_{r, m}
 \cong  h_{S*} \sBPS_{r, m'}.
\]
\end{cor}

We now prove the cohomological integrality theorem for Higgs bundles.
Recall that we have the following diagram:

\[
\xymatrix@C=50pt{
\fM_{X}^{\semist}(\mu) \ar[d]^-{p_X} \ar[r]^-{\pi}
& \fM_{S}^{\semist}(\mu) \ar[d]^-{p_S} \\
M_{X}^{\semist}(\mu) \ar[r]^-{\bar{\pi}}
& M_{S}^{\semist}(\mu).
}
\]

For a rational number $\mu$, we write
\[
\sBPS_{\mu} \coloneqq \bigoplus_{\frac{m}{r}=\mu} \sBPS_{r, m}.
\]
\begin{thm}\label{thm:intHiggs2}
The monodromic mixed Hodge module $\sBPS_{\mu}$ is contained in $\MHM(M^{\ss}_S(\mu))$, i.e., it has a trivial monodromy operator.
Further, we have an isomorphism 
\begin{equation}\label{eq:intHiggs2}
\bigoplus_{\frac{m}{r}=\mu} 
\cH({p_S}_* \bD \bQ_{\fM_{S}^{\semist}(r, m)}) \otimes \bL^{r^2(g-1)}
\cong \Sym_{\boxtimes_\oplus}\left( 
\mH^*(\mathrm{B}\bC^*)  \otimes \sBPS_{\mu}
\right)
\end{equation}
in $D^{\geq, lf}(\MHM(M^{\ss}_S(\mu)) )$.
\end{thm}

\begin{proof}
Proposition \ref{prop:decompthm} and Proposition \ref{prop:intHiggs} imply isomorphisms 
\begin{align*}
    \cH((p_S \circ \pi)_* \varphi_{\fM_{X}^{\semist}(\mu)}^{\mmhm}) 
    &\cong \cH((\bar{\pi} \circ p_X)_* \varphi_{\fM_{X}^{\semist}(\mu)}^{\mmhm}) \\
    &\cong \cH(\bar{\pi}_* \cH({p_X}_* \varphi_{\fM_{X}^{\semist}(\mu)}^{\mmhm})) \\
    &\cong \cH  \left( \bar{\pi}_* \Sym_{\boxtimes_\oplus}\left( 
\mH^*(\mathrm{B}\bC^*)_{\vir} \otimes \varphi_{M_{X}^{\semist}(\mu)}^{\mmhm}\right) \right) \\
&\cong \Sym_{\boxtimes_\oplus}\left( 
\mH^*(\mathrm{B}\bC^*) \otimes \sBPS_{\mu}
\right).
\end{align*}
As we have seen in Remark \ref{rmk:monodromyfree}, the left-hand side is monodromy-free, hence so is the BPS sheaf.
The isomorphism \eqref{eq:intHiggs2} follows from the above isomorphism and an isomorphism
\[
\cH((p_S \circ \pi)_* \varphi_{\fM_{X}^{\semist}(\mu)}^{\mmhm}) \cong \bigoplus_{\frac{m}{r}=\mu} 
\cH({p_S}_* \bD \bQ_{\fM_{S}^{\semist}(r, m)}) \otimes \bL^{r^2(g-1)}
\]
which is a consequence of Proposition \ref{prop:dimredMHM2} and the equality
$\vdim \bs{\fM}_{S}^{\semist}(r, m) = 2r^2(g-1)$.
\end{proof}

 \begin{cor}
 The  mixed Hodge module $\sBPS_{\mu}$ is pure.
 \end{cor}

 \begin{proof}
 The above theorem implies that there exists an embedding
 \[
 \sBPS_{r, m} \hookrightarrow 
 \cH({p_S}_* \bD \bQ_{\fM_{S}^{\semist}(r, m)}) \otimes \bL^{r^2(g-1)}.
 \]
 The purity of the right-hand side is proved in \cite[Proposition 7.20]{Dav21}, so we obtain the claim.
 \end{proof}

\begin{ex} \label{ex:ICvsBPS}
Assume that $(r, m)$ is coprime,
in which case $\fM_{S}^{\semist}(r, m)$ is smooth and $p_S$ is a $\bC^*$-gerbe. In this case, we have an isomorphism
\[
\cH^0({p_S}_* \bD \bQ_{\fM_{S}^{\semist}(r, m)}) \otimes \bL^{r^2(g-1)} 
\cong \sIC_{ M_{S}^{\semist}(r, m)}.
\]
Therefore we have an isomorphism
\[
\sIC_{ M_{S}^{\semist}(r, m)} \cong \sBPS_{r, m} .
\]
In particular, for coprime pairs $(r, m)$ and $(r, m')$, Corollary \ref{cor:chiBPS} implies an isomorphism
\[
 h_{S*} \sIC_{ M_{S}^{\semist}(r, m)}
 \cong  h_{S*} \sIC_{ M_{S}^{\semist}(r, m')}.
\]

Now let $(r, m)$ be a non-coprime pair.
It follows from \cite[Theorem 11.1]{Sim94} and \cite[Theorem 5.11]{Dav21} that $M_{S}^{\semist}(r, m)$ is normal.
The connectedness of $M_{S}^{\semist}(r, m)$ is proved in \cite[Claim 3.5 (iii)]{DP12}.
Therefore the moduli space $M_{S}^{\semist}(r, m)$ is irreducible.
Then using \cite[Theorem 6.6]{Dav21}, we can construct an inclusion
\begin{equation} \label{eq:ICinBPS}
\sIC_{M_{S}^{\semist}(r, m)} \hookrightarrow \sBPS_{r, m}
\end{equation}
but it is not necessary an isomorphism (see \S \ref{ssec:ex}).
\end{ex}

Write $B_S^*$ for the disjoint union of all the Hitchin bases (i.e. $B_S^*$ is the moduli space of all one-dimensional cycles on $S$).
We let $\oplus_B \colon B_S^* \times B_S^* \to B_S^*$ and $+ \colon \bN \times \bN \to \bN$ denote the canonical monoid structures.
The following statement is a direct consequence of Theorem \ref{thm:intHiggs2} and Lemma \ref{lem:decompthm}.

\begin{cor}
We have isomorphisms
\begin{align*}
\bigoplus_{\frac{m}{r}=\mu} 
\cH((h_S \circ p_S)_* \bD \bQ_{\fM_{S}^{\semist}(r, m)}) \otimes \bL^{r^2(g-1)}
&\cong \Sym_{\boxtimes_{\oplus_B}}\left( 
\mH^*(\mathrm{B}\bC^*)  \otimes {h_S}_* \sBPS_{\mu}
\right) \\
\bigoplus_{\frac{m}{r}=\mu} 
\HBM({\fM_{S}^{\semist}(r, m)}) \otimes \bL^{r^2(g-1)}
&\cong \Sym_{\boxtimes_+}\left( 
\mH^*(\mathrm{B}\bC^*)  \otimes \BPS_{\mu}
\right) 
\end{align*}
in $D^{\geq, lf}(\MHM(B_S^*) )$ and $D^{\geq, lf}(\MHM(\bN) )$ respectively.
\end{cor}

Combining the above corollary and the $\chi$-independence theorem for BPS cohomology (= Corollary \ref{cor:chiHiggs}), we obtain the following $\chi$-independence theorem for the Borel--Moore homology:

\begin{cor}\label{cor:chiBM}
Let $r, m, m'$ be integers such that $r>0$ and $\gcd(r, m) = \gcd(r, m')$.
Then there exist isomorphisms
\begin{align*}
 \cH((h_S \circ p_S)_* \bD \bQ_{\fM_{S}^{\semist}(r, m)}) &\cong \cH((h_S \circ p_S)_* \bD \bQ_{\fM_{S}^{\semist}(r, m')}), \\
 \HBM({\fM_{S}^{\semist}(r, m)}) &\cong \HBM({\fM_{S}^{\semist}(r, m')})
\end{align*}
in $D^{\geq, lf}(\MHM(B_S) )$ and $D^{\geq, lf}(\MHM(\pt) )$ respectively.
\end{cor}

\begin{rmk}
Based on P = W conjecture, it is conjectured in \cite{FSY21} that there exists an isomorphism of intersection cohomology groups
\[
\IH({M_{S}^{\semist}(r, m)}) \cong \IH({M_{S}^{\semist}(r, m')})
\]
preserving the perverse filtration for $r, m, m'$ such that  $\gcd(r, m) = \gcd(r, m')$.
At present we do not know how to prove this conjecture. However, once Davison's conjecture \cite[Conjecture 7.7]{Dav20} on the structure of the BPS sheaf is established, it would be possible to deduce the $\chi$-independence for intersection cohomology from the $\chi$-independence for BPS cohomology (= Corollary \ref{cor:chiHiggs}).
\end{rmk}

\subsection{An example: $g=2, r=2$}\label{ssec:ex}
Here we give an example where 
the intersection cohomology and the BPS cohomology are different. 
Let $C$ be a smooth projective curve of genus $2$, 
and put $S \coloneqq \Tot_C(\omega_C)$. 
We consider the moduli space $M_S(2, 0)$. 
By taking the cohomology of the inclusion \eqref{eq:ICinBPS}, 
we have an inclusion 
\begin{equation} \label{eq:IHinBPS20}
\IH(M_S(2, 0) ) \hookrightarrow \BPS_{2, 0}. 
\end{equation}

We will check that the above inclusion is 
not an isomorphism. 
Note that by Corollary \ref{cor:chiBPS} 
and Example \ref{ex:ICvsBPS}, 
we have an isomorphism 
\begin{equation} \label{eq:BPS20}
    \BPS_{2, 0} \cong \BPS_{2, 1} 
    \cong \IH(M_S(2, 1) ). 
\end{equation}

We denote by 
\[
\Phi_{\IC}(t) \coloneqq \sum_{i \in \bZ} 
h^i(\IH(M_S(2, 0) ) )t^i, \quad 
\Phi_{\BPS}(t) \coloneqq \sum_{i \in \bZ} 
h^i(\BPS_{2, 0} )t^i. 
\]
By \cite[Theorem 1.2]{Fel21} 
and \cite[Exercise 4.1]{Ray18}, 
we have 
\begin{equation} \label{eq:Ppolys}
\begin{aligned}
&\Phi_{\IC}(t)
=t^{-10}(
t+1
)^4
(
2t^6+2t^4+t^2+1
), \\
&\Phi_{\BPS}(t)=t^{-10}
(t+1 )^4
(2t^6+4t^5+2t^4+4t^3+t^2+1 ). 
\end{aligned}
\end{equation}
For the formula of $\Phi_{\BPS}$, 
we used the isomorphisms \eqref{eq:BPS20}. 
Note that the term $t^{-10}$ appears by our shift convention 
so that the intersection and the BPS complexes 
are perverse sheaves, together with the fact 
$\dim M_S(2, 0)=10$. 
From the formulas \eqref{eq:Ppolys}, it is obvious that 
the inclusion \eqref{eq:IHinBPS20} is not an isomorphism. 

On the other hand, 
assume that \cite[Conjecture 7.7]{Dav20} is true. 
Then we have an isomorphism 
\[
\sBPS_{2, 0} \cong \sIC_{M_S(2, 0) } 
\oplus \wedge_{\oplus}^2(\sIC_{M_S(1, 0) } ), 
\]
where $\wedge^2(-)$ is the graded wedge product 
in the category of graded vector spaces. 
Since the moduli space $M_S(1, 0)$ is 
isomorphic to the cotangent bundle of 
the Jacobian $\Jac(C)$ of $C$, 
we have 
\[
\IH(M_S(1, 0) ) 
\cong \mH^*(\Jac(C) )[4] 
=\bQ[4] \oplus \bQ^4[3] \oplus 
\bQ^6[2] \oplus \bQ^4[1] \oplus \bQ, 
\]
and hence we conclude that 
\begin{align*}
\sum_{i \in \bZ}h^i(
\wedge^2(\IH(M_S(1, 0) ) )
)t^i
&=t^{-7}(t+1 )^4
(t^2+1 ) \\
&=\Phi_{\BPS}(t)-\Phi_{\IC}(t).
\end{align*}
This computation gives 
an evidence of \cite[Conjecture 7.7]{Dav20}. 
At the same time, we can see that 
the $\chi$-independence for the intersection cohomology 
does not necessarily hold when $\gcd(r, m) \neq \gcd(r, m')$. 

\appendix

\section{Shifted symplectic structure and vanishing cycles}\label{app:tech}

In this appendix, we briefly recall the theory of shifted symplectic geometry and prove some technical lemmas including Proposition \ref{prop:joyce=van}.

\subsection{Shifted symplectic structures}
\label{sec:shifted}
We recall the notion of shifted symplectic structures introduced in \cite{PTVV13}.
Let $\bs{\fX}$ be a derived Artin stack.
We define the space of \textbf{$n$-shifted $p$-forms} $\cA^p(\bs{\fX}, n) \in \bS$ by
\[
\cA^p(\bs{\fX}, n) \coloneqq \Map(\cO_{\bs{\fX}}, \wedge^p \bL_{\bs{\fX}}[n] ).
\]
We can also define the space of \textbf{$n$-shifted closed $p$-forms} $\cA^{p, \cl}(\bs{\fX}, n) \in \bS$ (see \cite[Definition 1.12]{PTVV13}). It satisfies the \'etale descent and for a  connective commutative differential graded algebra $A$
we have an equivalence
\[
\cA^{p, \cl}(\DSpec A, n) \simeq \left|\prod_{i \geq 0} \wedge^{p + i} \bL_A [-i + n], d + \ddr \right|,
\]
where $d$ is the internal differential, $\ddr$ is the de Rham differential, and $|-|$ is the geometric realization functor.
The space of $n$-shifted $p$-forms and the space of $n$-shifted closed $p$-forms is functorial with respect to morphisms between derived Artin stacks, i.e., if we are given a morphism $\bs{f} \colon \bs{\fX} \to \bs{\fY}$, there exist natural maps
\begin{align*}
\bs{f}^{\star} &\colon \cA^{p}(\bs{\fY}, n) \to  \cA^{p}(\bs{\fX}, n), \\
\bs{f}^{\star} &\colon \cA^{p, \cl}(\bs{\fY}, n) \to  \cA^{p, \cl}(\bs{\fX}, n).
\end{align*}
We have a natural forgetful map 
\[
\pi \colon \cA^{p,\cl}(-, n) \to \cA^{p}(-, n)
\]
and the de Rham differential map
\[
\ddrcl \colon \cA^{p}(-, n) \to \cA^{p+1,\cl}(-, n).
\]

\begin{defin}
An $n$-shifted closed 2-form $\omega_{\bs{\fX}}$ is called an \textbf{$n$-shifted symplectic form} if its underlying $n$-shifted $2$-form is non-degenerate, i.e., the natural map
\[
\pi(\omega_{\bs{\fX}}) \cdot  \colon \bL_{\bs{\fX}}^{\vee} \to  \bL_{\bs{\fX}}[n]
\]
is an equivalence.
\end{defin}
In this paper, we are only interested in $(-1)$-shifted symplectic structures.

\begin{ex} 
\label{ex:CYderived}
Let $X$ be a Calabi--Yau threefold, i.e., a three dimensional smooth variety with trivial canonical bundle.
Then the derived moduli stack $\bs{\fM}_X$ of compactly supported coherent sheaves on $X$ carries a canonical $(-1)$-shifted symplectic structure. See \cite[Theorem 0.1]{PTVV13} and \cite[Main Theorem]{BDII}.
\end{ex}

\begin{ex}\label{ex:shiftedcot}
Let $\bs{\fY}$ be a derived Artin stack and 
\[
\bfT^*[n]\bs{\fY} \coloneqq \DSpec_{\bs{\fY}}(\Sym (\bL_{\bs{\fY}}^\vee[-n]))
\]
be its $n$-shifted cotangent stack.
Let $\lambda \in \cA^1(\bfT^*[n]\bs{\fY}, n)$ be the tautological $1$-form.
Then it is shown in \cite[Proposition 1.21]{PTVV13} and \cite[Theorem 2.2]{Cal19} that the $n$-shifted closed $2$-form $\ddrcl \lambda$ is shifted symplectic.
\end{ex}

\begin{ex} \label{ex:crit1}
Let $U = \Spec A$ be a smooth affine scheme which admits an \'etale coordinate $(x_1, \ldots, x_n)$, and $f \colon U \to \bA^1$ be a regular function.
Let $B$ be the cdga defined by the Koszul complex
\[
B \coloneqq (\cdots \to \Omega_A^{\vee} \xrightarrow[]{\ddr f} A).
\]
Then $\DSpec B$ is equivalent to the derived critical locus $\DCrit(f)$. We let $y_i \in B^{-1}$ be the element of degree $-1$ corresponding to $\partial/\partial x_i$.
Then the $(-1)$-shifted closed $2$-form
\[
\omega \coloneqq (\ddr x_1 \wedge \ddr y_1 + \cdots + \ddr x_n \wedge \ddr y_n, 0, 0,\ldots) \in \cA^{2, \cl}(\DSpec B, -1)
\]
defines a $(-1)$-shifted symplectic structure on $\DCrit(f)$.

It is shown in \cite[Theorem 5.18]{BBJ19} that any $(-1)$-shifted symplectic derived scheme is Zariski locally of this form.
\end{ex}

Now we discuss the canonical $(-1)$-shifted symplectic structure on the derived critical locus of a function on a general derived Artin stack. To do this, we need to recall the notion of Lagrangian structures.

\begin{defin}
Let $(\bs{\fX}, \omega_{\bs{\fX}})$ be an $n$-shifted symplectic derived Artin stack and $\bs{\tau} \colon \bs{L} \to \bs{\fX}$ be a morphism of derived Artin stacks.
An \textbf{isotropic structure} is a path from $0$ to $\bs{\tau}^{\star} \omega_{\bs{\fX}}$ in $\cA^{2, \cl}(\bs{L}, n)$. An isotropic structure $\eta$ is called a \textbf{Lagrangian structure} if it induces an equivalence 
\[
\bL_{\bs{L}}^{\vee} \simeq \bL_{\bs{\tau}}[n-1].
\]
See \cite[\S 2.2]{PTVV13} for the detail.
\end{defin}

\begin{ex}\label{ex:Lagrangian}
\begin{itemize}
    
    \item[(1)] Let $\bs{\fY}$ be a  derived Artin stack and $\lambda \in \cA^1(\bfT^*[n]\bs{\fY}, n)$ be the tautological $1$-form. Then $\lambda |_{\bs{\fY}}$ is naturally equivalent to zero hence so is $\ddrcl \lambda |_{\bs{\fY}}$. Therefore the zero section map $\bs{\fY} \to \bfT^*[n]\bs{\fY}$ carries a natural isotropic structure. It is shown in \cite[Theorem 2.2]{Cal19} that this isotropic structure is a Lagrangian structure.
    
    \item[(2)] Let $\bs{\fY}$ and $\lambda$ be as above, and take a function $f \in \Gamma(\bs{\fY}, \cO_{\bs{\fY}}[n])$ of degree $n$. Let $\overline{\ddr f} \colon \bs{\fY} \to \bfT^*[n]\bs{\fY}$ be the map corresponding to the section $\ddr f \in \Gamma(\bs{\fY}, \bL_{\bs{\fY}}[n])$. Then the natural homotopy 
    \[
    (\overline{\ddr f})^{\star} \ddrcl \lambda \sim \ddrcl \circ \ddr f \sim 0
    \]
    defines an isotropic structure on $(\overline{\ddr f})$.
    It is shown in \cite[Theorem 2.15]{Cal19} that this isotropic structure is a Lagrangian structure.
\end{itemize}
\end{ex}

Let $\bs{\fX}$ be an $n$-shifted symplectic derived Artin stack and $\bs{\tau}_1 \colon \bs{L}_1 \to \bs{\fX}$ and $\bs{\tau}_2 \colon \bs{L}_2 \to \bs{\fX}$ be Lagrangians.
These Lagrangian structures define a loop in $\cA^{2, \cl}(\bs{L_1} \times_{\bs{\fX}} \bs{L}_2, n)$ hence a point in $\cA^{2, \cl}(\bs{L_1} \times_{\bs{\fX}} \bs{L}_2, n-1)$.
It is shown in \cite[Theorem 2.9]{PTVV13} that this $(n-1)$-shifted closed $2$-form is shifted symplectic.

\begin{ex}\label{ex:crit2}
Let $\bs{\fY}$ be a derived Artin stack and $f \in \Gamma(\bs{\fY}, \cO_{\bs{\fY}}[n])$ be a function of degree $n$. The derived critical locus $\DCrit(f)$ is defined to be the intersection 
\[
\DCrit(f) \coloneqq \bs{\fY} \times_{0,\, \bfT^*[n]\bs{\fY},\, \overline{\ddr f}} \bs{\fY}.
\]
Example \ref{ex:Lagrangian} and the above discussion implies that $\DCrit(f)$ carries a canonical $(n-1)$-shifted symplectic structure.
\end{ex}

The $(-1)$-shifted symplectic structure constructed in Example \ref{ex:crit1} is a special case of the above example:
\begin{lem}\label{lem:equivsymp}
Let $f \colon U \to \bA^1$ be a regular function on a smooth affine scheme.
Assume that $U$ admits a global \'etale coordinate.
Then the $(-1)$-shifted symplectic structure on $\DCrit(f)$ constructed in Example \ref{ex:crit1} is equivalent to the $(-1)$-shifted symplectic structure constructed in Example \ref{ex:crit2}.
\end{lem}

\begin{proof}
We write $U = \Spec A$ and take a global \'etale coordinate $x_1, x_2, \ldots, x_n \in A$.
We let $B$ be the cdga appeared in Example \ref{ex:crit1} whose underlying graded algebra is $A[y_1, \ldots , y_n]$. As we have seen in Example \ref{ex:crit1}, $\DSpec B$ gives a model for $\DCrit(f)$. 
Consider the element 
\[
\alpha = \sum_i y_i \cdot \ddr x_i \in \Omega_B^{-1}.
\]
Then we have an identity $d \alpha = \ddr f$
 which corresponds to the natural homotopy 
\[
\ddr f|_{\DCrit(f)} \sim 0
\]
in $\mcA^1(\DCrit(f), 0)$. 
Therefore the element 
\[
\ddrcl\alpha = (\sum_i \ddr y_i \cdot \ddr x_i, 0, \ldots)
\]
corresponds to the $(-1)$-shifted symplectic structure constructed in Example \ref{ex:crit2}.
\end{proof}

Now we discuss the relation of the $(-1)$-shifted symplectic structure and the d-critical structure.
Let $(\bs{\fX}, \omega)$ be a $(-1)$-shifted symplectic derived Artin stack. Then it is shown in \cite[Theorem 3.18(a)]{BBBBJ15} that the classical truncation $\fX = t_0(\bs{\fX})$ carries a natural d-critical structure $s$. We now recall some of its basic properties.

Firstly assume that $\bs{\fX}$ is a derived scheme and write $\bs{X} = \bs{\fX}$ and $X = \fX$. Take an open embedding $\bs{\iota} \colon \DCrit(f) \hookrightarrow \bs{X}$ where $f$ is a regular function on a smooth scheme $U$ such that $f |_{\Crit(f)^{\red}} =  0$ and $U$ has a global \'etale coordinate. We equip $\DCrit(f)$ with the $(-1)$-shifted symplectic structure constructed in Example \ref{ex:crit1} and assume that $\bs{\iota}$ preserves the $(-1)$-shifted symplectic structures.
We let $R$ denote the image of $t_0(\bs{\iota})$ and $i \colon R \hookrightarrow U$ denote the natural inclusion. Then $(R, U, f, i)$ defines a d-critical chart of $(X, s)$, see \cite[Theorem 6.6]{BBJ19}.

Now we remove the assumption that $\bs{\fX}$ is a derived scheme.
Take a smooth morphism $\bs{q} \colon \bs{T} \to \bs{\fX}$. Assume that there exist a morphism between derived schemes $\bs{\tau} \colon \bs{T} \to \widehat{\bs{T}}$, a $(-1)$-shifted symplectic structure $\omega_{\widehat{\bs{T}}}$ on $\widehat{\bs{T}}$,
and an equivalence $\bs{q}^{\star} \omega \sim \bs{\tau}^{\star}\omega_{\widehat{\bs{T}}}$. Then there exists an equality 
\begin{equation}\label{eq:dcritst}
    t_0(\bs{q})^{\star}s = t_0(\bs{\tau})^{\star}s_{\widehat{T}}
\end{equation}
of d-critical structures, 
where $s_{\widehat{T}}$ is the d-critical structure on $\widehat{T} = t_0(\widehat{\bs{T}})$ induced from the $(-1)$-shifted symplectic structure $\omega_{\widehat{\bs{T}}}$. See \cite[Theorem 4.6]{Kin21} for the proof.

It is shown in \cite[Theorem 3.18(b)]{BBBBJ15} that there exists a natural isomorphism
\begin{equation}\label{eq:cancot}
K_{\fX, s}^{\vir} \cong \det(\bL_{\bs{\fX}} |_{\fX^{\red}}).
\end{equation}

\subsection{Proof of Proposition \ref{prop:joyce=van}}\label{ssec:joyvan}

Here we give the proof of Proposition \ref{prop:joyce=van}.

\begin{lem}
Let $q \colon U \to \fU$ be a smooth morphism from a smooth scheme.
We let $s_X$ denote the canonical d-critical structure on $X \coloneqq \Crit(f \circ q)$. Then there exists an equality 
$(q |_X)^{\star} s = s_X$.
\end{lem}

\begin{proof}
We let $\omega_{\DCrit(f)}$ (resp. $\omega_{\DCrit(f \circ q)}$) denote the natural $(-1)$-shifted symplectic structure on $\DCrit(f)$ (resp. $\DCrit(f \circ q)$).
Consider the following diagram of derived Artin stacks:
\[
\xymatrix{
\DCrit(f) \times_{\fU} U \ar[r]^-{\bs{\iota}} \ar[d]^-{\bs{q}_0}  
&  \DCrit(f \circ q) \\
\DCrit(f),
& 
{}
}
\]
where the map $\bs{q}_0$ is induced by $q$ and $\bs{\iota}$ is the natural map which is identity on the truncation.
We claim that there is an equivalence of $(-1)$-shifted symplectic structures: 
\begin{equation}\label{eq:critstack}
\bs{q}_0^{\star} \omega_{\DCrit(f)} \sim \bs{\iota}^{\star} \omega_{\DCrit(f \circ q)}.
\end{equation}

We have a natural homotopy 
\[
\alpha \colon \ddr(f)|_{\DCrit(f)} \sim 0
\]
in $\mathcal{A}^1(\DCrit(f), 0)$.
By definition, the symplectic form $\omega_{\DCrit(f)}$ corresponds to the loop 
\[
0 \sim \ddrcl \circ  \ddr(f)|_{\DCrit(f)} \sim 0
\]
in $\mathcal{A}^{2, \cl}(\DCrit(f), 0)$, 
where the first homotopy is defined by the equivalence $\ddrcl \circ \ddr \sim 0$ and the latter homotopy is defined by $\ddrcl \alpha$.
Therefore the closed $(-1)$-shifted $2$-form $\bs{q}_0^{\star} \omega_{\DCrit(f)}$ corresponds to the loop 
\[
0 \sim \ddrcl\circ  \ddr(f \circ q)|_{\DCrit(f) \times_{\fU} U} \sim 0
\]
in $\mathcal{A}^{2, \cl}(\DCrit(f) \times_{\fU} U, 0)$. 
A similar argument shows that $\bs{\iota}^{\star} \omega_{\DCrit(f \circ q)}$ has the same description, hence we obtain the equivalence \eqref{eq:critstack}.

Combining this equivalence and the equality \eqref{eq:dcritst}, we obtain the desired equality.
\end{proof}

\begin{lem} \label{lem:natori}
There exists a natural orientation of $(\fX, s)$
\[
o \colon K_{\fU}^{\otimes {2}} |_{\fX^{\red}} \cong K_{\fX, s}^{\vir}.
\]
\end{lem}

\begin{proof}
Take a smooth surjective morphism $q \colon U \to \fU$ and write $X = \Crit(f \circ q)$.
Consider the following composition
\begin{align*}
o_q \colon q^* K_{\fU}^{\otimes {2}}|_{X^{\red}}
&\cong (K_{U}^{\otimes {2}} \otimes \det(\Omega_{U/\fU})^{\otimes {-2}})|_{X^{\red}} \\
& \cong K_{X, s_X} \otimes \det(\Omega_{X/\fX})^{\otimes {-2}}|_{X^{\red}} \\
&\cong (q|_{X^{\red}})^* K_{\fX, s}, 
\end{align*}
where we set $s_X = (q|_X)^{\star}s$.
We claim that the isomorphism $o_q$ descends to an orientation for $(\fX, s)$.
To do this, take an \'etale surjective morphism $\eta \colon V \to U \times_{\fU} U$ from a scheme $V$. We let $\pr_i \colon U \times_{\fU} U$ be the $i$-th projection for $i= 1, 2$.
Write $Y = \Crit(f \circ q \circ \pr_1 \circ \eta)$ and define
\[
o_{q \circ \pr_i \circ \eta} \colon ((q \circ \pr_i \circ \eta)^* K_{\fU}^{\otimes ^{2}}) |_{Y^{\red}} \to (q \circ \pr_i \circ \eta |_{Y^{\red}})^* K_{\fX, s}
\]
in the same manner as $o_q$.
It is enough to prove the commutativity of the following diagram for each $i= 1, 2$:
\[
\xymatrix@C=50pt{
((q \circ \pr_i \circ \eta)^* K_{\fU}^{\otimes ^{2}}) |_{Y^{\red}} \ar[r]^-{o_{q \circ \pr_i \circ \eta}} \ar[d]
& (q \circ \pr_i \circ \eta |_{Y^{\red}})^* K_{\fX, s} \ar[d]   \\
(\pr_i \circ \eta)^{\red, *}(q^* K_{\fU}^{\otimes {2}}|_{X^{\red}}) \ar[r]^-{(\pr_i \circ \eta)^{\red, *} o_q}
& (\pr_i \circ \eta)^{\red, *} (q|_{X^{\red}})^* K_{\fX, s}. 
}
\]
This follows from the commutativity of the diagram \eqref{eq:diagUp}.

\end{proof}

\begin{proof}[Proof of Proposition \ref{prop:joyce=van}]
We keep the notation as in the proof of the previous lemma.
Let $o_X$ and $o_Y$ be natural orientations on $(X, q^{\star}s)$ and $(Y, \eta^{\star }\pr_{1}^{\star} q^{\star}s)$ coming from the descriptions as global critical loci.
The construction of the orientation $o$ in the previous lemma implies that we have the following natural commutative diagram of orientations
\begin{equation}\label{eq:commori}
\xymatrix@C=15pt@R=10pt{
 o_Y \ar[rr] \ar[rd] & & \eta^{\star}\pr_{i}^{\star} o_X \ar[ld] \\
& \eta^{\star} \pr_{i}^{\star} q^{\star} o &
}
\end{equation}
for each $i = 1, 2$.

Now define an isomorphism 
\[
\theta_q \colon q^* \varphi_{\fX, s, o} \cong q^*\varphi_f(\bQ_{\fU}[\dim \fU])
\]
by the following composition
\begin{align*}
        q^* \varphi_{\fX, s, o} 
 &\cong \varphi_{X, q^{\star}s, q^{\star} o}[- \dim q] \\
 &\cong \varphi_{X, s_X, o_X}[-\dim q] \\
 &\cong \varphi_{f \circ q}(\bQ_{U}[\dim U - \dim q]) \\
 &\cong q^* \varphi_{f }(\bQ_{\fU}[\dim \fU])
\end{align*}
for $i = 1, 2$.
Here the third isomorphism follows from Lemma \ref{lem:equivsymp}
 and Example \ref{ex:vanloc}. Similarly we can define an isomorphism 
\[
\theta_{q \circ \pr_i \circ \eta} \colon (q \circ \pr_i \circ \eta)^*\varphi_{\fX, s, o} \cong (q \circ \pr_i \circ \eta)^* \varphi_{f }(\bQ_{X}[\dim \fX]).
\]
The commutativity of the diagrams \eqref{eq:commori},
\eqref{eq:commvan}, and \eqref{eq:commvanpull} implies an equality
\[
({\pr_i \circ \eta})^* \theta_{q} = \theta_{q \circ \pr_i \circ \eta}
\]
hence $\theta_q$ descends to the desired isomorphism.
\end{proof}

\section{Proof of the support lemma}\label{app:supplem}

We will give the proof of the support lemma (= Proposition \ref{prop:supplem}) here.

Fix positive integers $r_1, r_2$ and integers $m_1, m_2$ such that $\mu = m_1/r_1 = m_2/r_2$ holds.
Write $r = r_1 + r_2$ and $m= m_1 + m_2$.
Define $\bs{\fW} \subset \bs{\fM}_{X}(r_1, m_1) \times \bs{\fM}_{X}(r_2, m_2)$ to be the substack 
consisting of pairs $([E], [F])$ such that  $p(\Supp E) \cap p(\Supp F) = \emptyset$ where $p \colon X = S \times \bA^1 \to \bA^1$ is the projection.
Define a map $\bs{w} \colon \bs{\fW} \to \bs{\fM}_{X}(r, m)$ by taking direct sum.
The map $\bs{w}$ is an \'etale map.

\begin{lem}
Let $\omega_1$ (resp. $\omega_2$, $\omega$) be the $(-1)$-shifted symplectic structure on $\bs{\fM}_{X}(r_1, m_1)$ (resp. $\bs{\fM}_{X}(r_2, m_2)$, $\bs{\fM}_{X}(r, m)$).
Then there exists an equivalence
\[
(\omega_1 \boxplus \omega_2)|_{\bs{\fW}} \simeq \bs{w}^{\star} \omega.
\]
\end{lem}

\begin{proof}
It follows from \cite[Corollary 6.5]{BDII} that there exists a Lagrangian structure on the morphism
\[
(\bs{\iota}, \bs{w}) \colon \bs{\fW} \to \bs{\fM}_{X}(r_1, m_1) \times \bs{\fM}_{X}(r_2, m_2) \times \bs{\fM}_{X}(r, m)
\]
where $\bs{\iota}$ is the natural inclusion and we equip $\bs{\fM}_{X}(r_1, m_1) \times \bs{\fM}_{X}(r_2, m_2) \times \bs{\fM}_{X}(r, m)$ with the $(-1)$-shifted symplectic structure $\omega_1 \boxplus \omega_2 \boxplus (- \omega)$. The Lagrangian structure induces the desired equivalence.

\end{proof}

 For an open subset $U \subset \bC$ in the analytic topology, 
 we define  $\fM_{X}^{\semist}(r, m)^{U} \subset \fM_{X}^{\semist}(r, m)$ to be the complex analytic open substack consisting of points corresponding to sheaves whose supports are contained in $S \times U \subset X$.
 The following statement is a straightforward consequence of the above lemma.

 \begin{cor}\label{cor:compdcrit}
 Let $U_1, U_2 \subset \bC$ be disjoint open subsets in the analytic topology. 
 Consider the following open immersion
 \[
 w_{U_1, U_2} \colon \fM_{X}^{\semist}(r_1, m_1)^{U_1} \times \fM_{X}^{\semist}(r_2, m_2)^{U_2} \to \fM_{X}^{\semist}(r, m)^{U_1 \coprod U_2}
 \]
 induced from $\bs{w}$.
 We let $s_i$ denote the d-critical structure on $\fM_{X}^{\semist}(r_i, m_i)^{U_i}$ and $s$ denote the d-critical structure on the right-hand side. Then we have an equality $w_{U_1, U_2}^{\star}(s) = s_1 \boxplus s_2$.
 \end{cor}
 
 We now want to prove that the map $w_{U_1, U_2}$ preserves the canonical orientation following the idea of the proof of Proposition \ref{prop:compori}.
 Let $W$ be the image of $t_0(\bs{\fW})$ along the map 
 \[
 \fM_{X}^{\semist}(r_1, m_1) \times \fM_{X}^{\semist}(r_2, m_2)  \to M_{X}^{\semist}(r_1, m_1) \times M_{X}^{\semist}(r_2, m_2).
 \]
 Write $(r_1, m_1) = (k_1 r_0 , k_1 m_0)$ and $(r_2, m_2) = (k_2 r_0 , k_2 m_0)$ where $(r_0, m_0)$ is coprime. Define an open subspace
 \[
 \bA_{W} \subset \Sym^{k_1}(\bA^1) \times \Sym^{k_2}(\bA^1)
 \]
 consisting of configurations $(P, Q) $ such that $P \cap Q = \emptyset$.
 There exists a natural map
 \[
 W^{\red} \hookrightarrow M_{X}^{\semist}(r_1, m_1)^{\red} \times M_{X}^{\semist}(r_2, m_2)^{\red} \to \Sym^{r_1}(\bA^1) \times \Sym^{r_2}(\bA^1)
 \]
 where the latter map is  induced by the projection $X = S \times \bA^1 \to \bA^1$.
 Note that the above map factors through the inclusion
 \[
 \bA_{W} \hookrightarrow \Sym^{k_1}(\bA^1) \times \Sym^{k_2}(\bA^1) \hookrightarrow \Sym^{r_1}(\bA^1) \times \Sym^{r_2}(\bA^1)
 \]
 where the latter map is the diagonal embedding.
 Therefore we obtain a surjective map
 \[
 \eta_W \colon W^{\red} \to \bA_{W}.
 \]

\begin{lem}\label{lem:funloc}
Let $f$ be an invertible regular function on $W^{\red}$. Then there exists a regular function $g$ on $\bA_W$ such that $f = g \circ \eta_W$ holds.
\end{lem}

\begin{proof}
We first claim that the function $f$ is constant along the reduced parts of fibers of $\eta_W$.
Take a configuration $(P, Q) \in \bA_W$.
Write $P \cup Q = \{p_1, \ldots p_t\}$ and we let $l_i$ be the multiplicity of $P \cup Q$ at $p_i$.
Then the fiber $\eta_W^{-1}((P, Q))$ is isomorphic to the scheme
\[
\prod_{i= 1}^{t} M_{S}^{\semist}(l_i r_0, l_i m_0).
\]
Therefore it follows from the proof of Proposition \ref{prop:compori} that $f$ is constant along $\eta_W^{-1}((P, Q))^{\red}$.
Therefore it is enough to prove that the map $\eta_W$ admits a section.

Take an arbitrary stable sheaf $E \in M_{S}^{\semist}(r_0, m_0)$.
Consider the map $s_W \colon \bA_{W} \to W^{\red}$ defined by
\[
 (P = \{p_1, \ldots p_t \}, Q  = \{q_1, \ldots q_s \}) \mapsto  (\bigoplus_{j}(i_{p_j, *} E)^{\oplus \mathrm{mult}(p_j)},  (\bigoplus_{j}(i_{q_j, *} E)^{\oplus \mathrm{mult}(q_j)}) 
\]
where $i_{p_j}$ denotes the embedding $S \times \{ p_j \} \hookrightarrow X$ and
$\mult(p_j)$ denotes the multiplicity of $p_j$ in $P$, and similarly for $i_{q_j}$ and $\mult(q_j)$.
We can see that $s_W$ is a section of $\eta_W$. Thus we obtain the claim.

\end{proof}

Let $s_{\fW}$ be the natural d-critical structure on $\fW = t_0(\bs{\fW})$.
Let $o_{\fW}\colon  M_1^{\otimes {2}} \cong K_{\fW, s_{\fW}}^{\vir}$ be the orientation on $\fW$ induced from the canonical orientation on $\fM_{X}^{\semist}(r_1, m_1) \times \fM_{X}^{\semist}(r_2, m_2)$ and 
$o_{\fW}'\colon  M_2^{\otimes {2}} \cong  K_{\fW, s_{\fW}}^{\vir}$ be the orientation on $\fW$ induced from the canonical orientation on
$\fM_{X}^{\semist}(r, m)$.
We have seen in the proof of Proposition \ref{prop:compori} that there exist trivializations $M_1 \cong \cO_{\fW^{\red}}$ and $M_2 \cong \cO_{\fW^{\red}}$.
Therefore the composition $(o_{\fW}')^{-1} \circ o_{\fW}$
defines an element
\[
\alpha \in \Gamma(\fW^{\red}, \cO_{\fW^{\red}}) \cong \Gamma(W^{\red}, \cO_{W^{\red}}).
\]

\begin{cor}\label{cor:compori2}
We use the notations as in Corollary \ref{cor:compdcrit}.
Assume that each connected component of $U_1$ and $U_2$ are homeomorphic to the disk.
Then there exists an isomorphism of orientations 
\[o_{\fW} |_{\fM_{X}^{\semist}(r_1, m_1)^{U_1} \times \fM_{X}^{\semist}(r_2, m_2)^{U_2}} \cong o_{\fW}' |_{\fM_{X}^{\semist}(r_1, m_1)^{U_1} \times \fM_{X}^{\semist}(r_2, m_2)^{U_2}}.
\]
\end{cor}

\begin{proof}
Let $M_{X}^{\semist}(r_1, m_1)^{U_1} \subset M_{X}^{\semist}(r_1, m_1)$ (resp. $M_{X}^{\semist}(r_2, m_2)^{U_2} \subset M_{X}^{\semist}(r_2, m_2)$) be the image of $\fM_{X}^{\semist}(r_1, m_1)^{U_1}$ (resp. $\fM_{X}^{\semist}(r_2, m_2)^{U_2}$) along the map $p_X$.
We need to show that $\alpha |_{M_{X}^{\semist}(r_1, m_1)^{U_1} \times M_{X}^{\semist}(r_2, m_2)^{U_2}}$ admits a square root.
Lemma \ref{lem:funloc} implies that there exists a regular function $g$ on $\bA_{W}$ such that $\alpha = g \circ \eta_W $ holds.
Let $\bA_{U_1, U_2} \subset \bA_{W}$ be an open subset consisting of configurations $(P, Q)$ such that $P \subset U_1$ and $Q \subset U_2$.
Since the image of $M_{X}^{\semist}(r_1, m_1)^{U_1} \times M_{X}^{\semist}(r_2, m_2)^{U_2}$ under the map $\eta_W$ is contained in $\bA_{U_1, U_2}$, we need to show that $g|_{\bA_{U_1, U_2}}$ admits a square root. But this follows from the simply connectedness of $\bA_{U_1, U_2}$.
\end{proof}

\begin{proof}[Proof of Proposition \ref{prop:supplem}]
The proof is almost identical to the proof of the support lemma for preprojective algebras \cite{Dav16}.
Take disjoint open subsets $U_1, U_2 \subset \bA^1$ whose connected components are homeomorphic to the disk.
It follows from Proposition \ref{prop:intHiggs} that there exists an isomorphism
\begin{align*}
\cH\left(({p_X}_* \varphi_{\fM_{X}^{\semist}(\mu)})|_{{M}_{X}^{\semist}(\mu)^{U_i}}\right) 
\cong 
\Sym_{\boxtimes_\oplus}\left( 
\mH^*(\mathrm{B}\bC^*)_{\vir} \otimes \varphi_{M_{X}^{\semist}(\mu)|_{{M}_{X}^{\semist}(\mu)^{U_i}}}
\right)
\end{align*}
for each $i=1, 2$. Then Thom--Sebastiani theorem \cite[Remark 5.23]{AB17}, Corollary \ref{cor:compdcrit} and Corollary \ref{cor:compori2} imply the following isomorphism

\begin{align*}
    & \cH\left(({p_X}_* \varphi_{\fM_{X}^{\semist}(\mu)})|_{{M}_{X}^{\semist}(\mu)^{U_1 \coprod U_2}}\right) \\
\cong & \cH\left(({p_X}_* \varphi_{\fM_{X}^{\semist}(\mu)})|_{{M}_{X}^{\semist}(\mu)^{U_1}}\right)  \boxtimes_{\oplus} \cH\left(({p_X}_* \varphi_{\fM_{X}^{\semist}(\mu)})|_{{M}_{X}^{\semist}(\mu)^{U_2}}\right) \\
\cong& \Sym_{\boxtimes_\oplus}\left( 
\mH^*(\mathrm{B}\bC^*)_{\vir} \otimes \varphi_{M_{X}^{\semist}(\mu)|_{{M}_{X}^{\semist}(\mu)^{U_1}}}
\right) \\
&\ \ \ \ \ \ \ \ \ \ \ \ \ \ \ \ \ \ \  \boxtimes_{\oplus} \Sym_{\boxtimes_\oplus}\left( 
\mH^*(\mathrm{B}\bC^*)_{\vir} \otimes \varphi_{M_{X}^{\semist}(\mu)|_{{M}_{X}^{\semist}(\mu)^{U_2}}}
\right) \\
\cong& \Sym_{\boxtimes_\oplus}\left( 
\mH^*(\mathrm{B}\bC^*)_{\vir} \otimes (\varphi_{M_{X}^{\semist}(\mu)|_{{M}_{X}^{\semist}(\mu)^{U_1}}} \oplus \varphi_{M_{X}^{\semist}(\mu)|_{{M}_{X}^{\semist}(\mu)^{U_2}}}) \right)
\end{align*}

On the other hand, Proposition \ref{prop:intHiggs} implies an isomorphism
\begin{align*}
\cH\left(({p_X}_* \varphi_{\fM_{X}^{\semist}(\mu)})|_{{M}_{X}^{\semist}(\mu)^{U_1 \coprod U_2}})\right)
\cong 
\Sym_{\boxtimes_\oplus}\left( 
\mH^*(\mathrm{B}\bC^*)_{\vir} \otimes \varphi_{M_{X}^{\semist}(\mu)|_{{M}_{X}^{\semist}(\mu)^{U_1 \coprod U_2}}}
\right).
\end{align*}
Therefore we obtain an isomorphism
\begin{equation}\label{eq:suppisom}
\varphi_{M_{X}^{\semist}(\mu)|_{{M}_{X}^{\semist}(\mu)^{U_1}}} \oplus \varphi_{M_{X}^{\semist}(\mu)|_{{M}_{X}^{\semist}(\mu)^{U_2}}} 
\cong \varphi_{M_{X}^{\semist}(\mu)|_{{M}_{X}^{\semist}(\mu)^{U_1 \coprod U_2}}}.
\end{equation}

Now take a point $[E] \in \Supp(\varphi_{M_{X}^{\semist}(\mu)})$ where $E$ is a polystable sheaf on $X$.
Assume that the support of $E$ is contained in $S \times (U_1 \coprod U_2)$ for some disjoint open subsets $U_1, U_2 \subset \bC$ (or equivalently, $[E] \in M_{X}^{\semist}(\mu)^{U_1 \coprod U_2}$).
Then the isomorphism \eqref{eq:suppisom} implies that the support of $E$ is contained in either of $U_1$ or $U_2$. Therefore there exists some $t \in \bC$ such that $\Supp(E) \subset S \times \{t\}$, which implies the proposition.

\end{proof}

\bibliographystyle{amsalpha}
\bibliography{biblio}

\end{document}